\documentclass[11pt]{article}
\usepackage{amssymb}
\usepackage{amsmath}
\usepackage{amsthm}
\usepackage{latexsym}
\usepackage[english]{babel}

\usepackage{amsmath,amssymb,amscd,enumerate,eucal,bm}
\usepackage{graphicx}
\usepackage{epstopdf}
\usepackage{amsmath}
\usepackage{amsfonts}
\usepackage{amssymb}
\usepackage{amsmath}
\usepackage{amsthm}
\usepackage{latexsym}
\usepackage{graphicx}
\usepackage{epstopdf}
\usepackage{float}
\usepackage{caption}
\usepackage{subfigure}
\usepackage{tikz}
\usetikzlibrary{matrix,arrows,decorations.pathmorphing}

\numberwithin{equation}{section}

\newcommand{\R}{\mathbb{R}}

\newcommand{\N}{\mathbb{N}}

\newcommand{\U}{U_\infty}
\newcommand{\Ps}{\Psi_\infty}

\newtheorem{theorem}{Theorem}[section]

\newtheorem{corollary}[theorem]{Corollary}

\newtheorem{lemma}[theorem]{Lemma}

\newtheorem{proposition}[theorem]{Proposition}
\newtheorem{remark}[theorem]{Remark}

\begin{document}

\title{Characterization of the Palais-Smale sequences for the conformal Dirac-Einstein problem and applications}

\author{Ali Maalaoui$^{(1)}$ \& Vittorio Martino$^{(2)}$}
\addtocounter{footnote}{1}
\footnotetext{Department of mathematics and natural sciences, American University of Ras Al Khaimah, PO Box 10021, Ras Al Khaimah, UAE. E-mail address:
{\tt{ali.maalaoui@aurak.ac.ae}}}
\addtocounter{footnote}{1}
\footnotetext{Dipartimento di Matematica, Universit\`a di Bologna,
piazza di Porta S.Donato 5, 40127 Bologna, Italy. E-mail address:
{\tt{vittorio.martino3@unibo.it}}}

\date{}
\maketitle

\vspace{5mm}

{\noindent\bf Abstract} {\small In this paper we study the Palais-Smale sequences of the conformal Dirac-Einstein problem. After we characterize the bubbling phenomena, we prove an Aubin type result leading to the existence of a positive solution. Then we show the existence of infinitely many solutions to the problem provided that the underlying manifold exhibits certain symmetries.}
\vspace{8mm}


\vspace{4mm}


\noindent

\vspace{4mm}
\tableofcontents

\noindent


\section{Introduction and statement of the results}

\noindent
We start by recalling the super-symmetric model consisting of coupling gravity with fermionic interaction. We fix a three dimensional closed (compact, without boundary) manifold $M$, then we define the energy functional $\mathcal{E}$ of this model by
$$\mathcal{E}(g,\psi)=\int_{M}R_{g}dv_{g}+\int_{M}\langle D_{g}\psi,\psi \rangle-\langle \psi,\psi\rangle  dv_{g},$$
where $g$ is a Riemannian metric on $M$, $\psi$ is a spinor in the spin bundle $\Sigma M$ on $M$, $R_g$ is the scalar curvature, $D_{g}$ is the Dirac operator and $\langle \cdot, \cdot \rangle$ is the compatible Hermitian metric on $\Sigma M$; we will give the precise definitions in the next section. The functional $\mathcal{E}$ generalizes the classical Hilbert-Einstein functional and it is invariant under the group of diffeomorphisms of $M$ as well; we address the reader to \cite{Belg, Fin, Kim}, where it was introduced and studied.\\
As in the classical case of the Hilbert-Einstein functional, since the group of diffeomorphisms is usually big, in first instance we restrict the functional to a fixed conformal class, namely given a Riemannian metric $g$, we set
$$[g]=\left\{u^{4}g;\; u\in C^{2,\alpha}(M)\right\} .$$
In this way, the energy functional reads as
\begin{equation}
E(u,\psi)=\frac{1}{2}\left(\int_{M}u L_{g}u+\langle D_{g} \psi,\psi \rangle -|u|^{2}|\psi|^{2} dv_{g}\right),
\end{equation}
where $L_{g}$ is the conformal Laplacian of the metric $g$.\\
This energy functional can also be seen as the three dimensional version of the super-Liouville equation investigated in \cite{J1,J2}, which is fundamental in the study of string fermions, see \cite{Fuk}.\\
By the first variation of the functional $E$, we see that its critical points satisfy the coupled system
\begin{equation}\label{el}
\left\{\begin{array}{ll}
L_{g} u=|\psi|^{2}u\\
& \text{on } M .\\
D_{g}\psi=|u|^{2}\psi
\end{array}
\right.
\end{equation}
Since the functional is conformally invariant, one expects compactness to be violated for this problem; moreover, due to the presence of the Dirac operator, it is strongly indefinite. For the later part, the authors studied an effective method, based on a homological approach \cite{M,MV,MV1}, for general functionals with this feature of being strongly indefinite; here we will focus on the first issue, that is the lack of compactness.\\

\noindent
We recall that a $C^{1}$ function $F$ satisfies the Palais-Smale condition (PS) if: for any sequence $x_{k}$ such that $F(x_{k})\to c$ and $\nabla F(x_{k})\to 0$ (such a sequence is then called a (PS) sequence), there exists a converging subsequence.\\
The (PS) condition is fundamental in the study of problems with variational structure as min-max theorems or Morse type methods, which rely heavily on this condition since it guaranties the convergence of the deformation flow. In several geometric problems though, such condition is violated, mainly because of the conformal invariance. We recall the widely investigated cases of prescribing curvatures as the Yamabe problem or the $Q$-curvature problem (see for instance \cite{L,B,Br,Ben,Ben2} and the references therein). In the previously stated problems, the lack of compactness is well understood and there is a specific characterization for the (PS) sequences. The first result in this paper concerns the study and the characterization of the (PS) sequences of the functional $E$, in particular we will show the following:
\begin{theorem}\label{first}
Let us assume that $M$ has a positive Yamabe constant $Y_g(M)$ and let $(u_{n},\psi_{n})$ be a Palais-Smale sequence for $E$ at level $c$. Then there exist $u_{\infty}\in C^{2,\alpha}(M)$, $\psi_{\infty}\in C^{1,\beta}(\Sigma M)$ such that $(u_{\infty},\psi_{\infty})$ is a solution of $(\ref{el})$, $m$ sequences of points $x_{n}^{1},\cdots, x_{n}^{m} \in M$ such that $\lim_{n\to \infty}x_{n}^{k}= x^{k}\in M$, for $k=1,\dots,m$ and $m$ sequences of real numbers $R_{n}^{1},\cdots, R_{n}^{m}$ converging to zero, such that:
\begin{itemize}
\item[i)]   $\displaystyle u_{n}=u_{\infty}+\sum_{k=1}^{m} v_{n}^{k}+o(1)$  in  $H^{1}(M)$,
\item[ii)]  $\displaystyle \psi_{n}=\psi_{\infty}+\sum_{k=1}^{m}\phi_{n}^{k}+o(1)$ in $H^{\frac{1}{2}}(\Sigma M)$,
\item[iii)] $\displaystyle E(u_{n},\psi_{n})=E(u_{\infty},\psi_{\infty})+\sum_{k=1}^{m}E_{\mathbb{R}^{3}}(U_\infty^{k},\Psi_\infty^{k})+o(1)$,
\end{itemize}
where
$$v_{n}^{k}=(R_{n}^{k})^{-\frac{1}{2}}\beta_{k}\sigma_{n,k}^{*}(U_\infty^{k}) ,$$
$$\phi_{n}^{k}=(R_{n}^{k})^{-1}\beta_{k}\sigma_{n,k}^{*}(\Psi_\infty^{k}) ,$$
with $\sigma_{n,k}=(\rho_{n,k})^{-1}$ and $\rho_{n,k}(\cdot)=exp_{x_{n}^k}(R_{n}^k \cdot)$ is the exponential map defined in a suitable neighborhood of $\R^{3}$.
Also, here $\beta_{k}$ is a smooth compactly supported function, such that $\beta_{k}=1$ on $B_{1}(x^{k})$ and $supp(\beta_{k})\subset B_{2}(x^{k})$ and $(U_\infty^{k},\Psi_\infty^{k})$ are solutions to our equations (\ref{el}) on $\mathbb{R}^3$ with its Euclidian metric $g_{\R^3}$.
\end{theorem}

\begin{remark}
The assumption on $M$ of having a positive Yamabe constant implies in particular that there are no harmonic spinors, namely the Dirac operator $D_g$ has no kernel: this will be used in the proof. In fact, by conformal invariance of the Dirac operator, the vanishing of the kernel is preserved by conformal change. So if the Yamabe constant is positive, then the conformal class of the metric contains a metric with positive scalar curvature, hence using the Schr\"{o}dinger-Lichnerowicz formula for this last metric and denoting by $\Delta_{\Sigma}$ the connection Laplacian, we have that
$$D_{g}^{2}=-\Delta_{\Sigma}+\frac{R_g}{4},$$
which implies the vanishing of the kernel of $D_g$.
\end{remark}

\noindent
Here $H^{1}(M)$ and $H^{\frac{1}{2}}(\Sigma M)$ are suitable Sobolev spaces on which the functional $E$ is well defined (see next section).\\

\noindent
Now, for non-trivial  $(u,\psi) \in H^1(M) \times H^{\frac{1}{2}}(\Sigma M)$, we define the functionals
$$\tilde{E}(u,\psi)=\frac{\left(\displaystyle\int_{M}uL_{g}udv_{g}\right)\left(\displaystyle\int_{M}\langle D_g \psi,\psi\rangle dv_{g}\right)}{\displaystyle\int_{M}|u|^{2}|\psi|^{2}dv_{g}} ,\quad I(\psi)=\frac{\displaystyle\int_{M}\langle D_g \psi,\psi\rangle dv_{g}}{\displaystyle\int_{M}|u|^{2}|\psi|^{2}dv_{g}}.$$
Also, we let $P^{-}$ be the projector on $H^{\frac{1}{2},-}$ (the negative space of $H^{\frac{1}{2}}(\Sigma M)$ according to the splitting given by the eigenspinors of the Dirac operator), so that for a given $\psi \in H^{\frac{1}{2}}(\Sigma M)$
$$P^{-}(\psi)=0 \Longleftrightarrow \int_{M}\langle \psi,\varphi \rangle dv_{g}, \quad \forall \varphi\in H^{\frac{1}{2},-} .$$
As for the Yamabe problem, we define a conformal constant
$$\tilde{Y}_g (M)=\inf\left\{\begin{array}{ll}
\tilde{E}(u,\psi); \text{ for  $(u,\psi)\in H^{1}(M)\setminus\{0\}\times H^{\frac{1}{2}}(\Sigma M)\setminus \{0\}$ s.t. }  I(\psi)> 0, \\
 P^{-}\left(D_g \psi-I(\psi)u^{2}\psi\right) =0\end{array}  \right\} .$$
Indeed, we first recognize that the constant $\tilde{Y}_g (M)$ only depends on the conformal class of the metric $g$, then we have an Aubin type result by comparing $\tilde{Y}_g (M)$ with the invariants on the sphere $S^3$ with its standard metric $g_0$. In particular we will show the following
\begin{theorem}\label{second}
It holds:
$$Y_g(M)\lambda^{+}_g(M)\leq \tilde{Y}_g(M) \leq Y_{g_0}(S^{3})\lambda^{+}_{g_0}(S^{3})=\tilde{Y}_{g_0}(S^{3}) .$$
Moreover, if
$$\tilde{Y}_g(M) < \tilde{Y}_{g_0}(S^{3}) ,$$
then problem $(\ref{el})$ has a non-trivial ground state solution.
\end{theorem}
\noindent
Here we have denoted by
$$\lambda^{+}_g(M):=\inf_{\tilde g \in [g]}\lambda_1(\tilde g)Vol_{\tilde g}(M)^{\frac{1}{3}}$$
the invariant as defined in \cite{Ginoux} and $\lambda_1(g)$ being the smallest positive eigenvalue of $D_g$ on $M$.\\
We recall that $\lambda^{+}_{g}(M)$ can be characterized as follows (see \cite{Am1,Am2}):
\begin{align}
\lambda^{+}_{g}(M)&=\inf_{\psi \in im_{C^{\infty}}D_{g}\setminus\{0\}}\frac{\displaystyle\left(\int_{M}|\psi|^{\frac{3}{2}}dv_{g}\right)^{\frac{4}{3}}}{\displaystyle\left|\int_{M}\langle \psi,D^{-1}_{g}\psi \rangle dv_{g}\right|}\notag\\
&=\inf_{\varphi \in W^{1,\frac{3}{2}}(\Sigma M)\setminus \{0\}} \frac{\displaystyle\left(\int_{M}|D_{g}\varphi|^{\frac{3}{2}}dv_{g}\right)^{\frac{4}{3}}}{\displaystyle\left|\int_{M} \langle \varphi, D_{g}\varphi \rangle dv_{g}\right|}, \notag
\end{align}
where $im_{C^{\infty}}D_{g}$ is the image of the operator $D_{g}:C^{\infty}(\Sigma M)\to C^{\infty}(\Sigma M)$.\\

\noindent
Finally, in the last section we will consider a three-dimensional closed manifold $M$ with an isometric group action $G$ acting on $M$, such that the orbits of $G$ have infinite cardinality and we will show that equations $(\ref{el})$ admit two infinite families of solutions on such a manifold.


\section{Conformally invariant operators, spaces of variations and splitting}

\noindent
In this section we will briefly recall some notations and properties of conformally invariant operators involved and we will give the definition of the Sobolev spaces that we are going to use.\\
Let $(M,g)$ be a closed (compact, without boundary) three dimensional Riemannian manifold, we define the conformal Laplacian acting on functions by
$$L_g u:=-\Delta_g u+\frac{1}{8}R_g u ,$$
where $\Delta_g$ is the standard Laplace-Beltrami operator and $R_g$ is the scalar curvature. The conformal invariance of $L_g$ reads as follows: if $\tilde g=g_u= u^4g$ is a metric in the conformal class of $g$, then we have
$$L_{\tilde g}f=u^{-5}L_g(uf) .$$
We will denote by $H^1(M)$ the usual Sobolev space on $M$, and we recall that by the Sobolev embedding theorems there is a continuous embedding $$H^1(M)\hookrightarrow L^p(M), \quad 1\leq p \leq 6 ,$$
which is compact if $1\leq p <6$.\\
Now let $\Sigma M$ the canonical spinor bundle associated to $M$ see \cite{F}, whose sections are simply called spinors on $M$.
This bundle is endowed with a natural Clifford multiplication $\text{Cliff}$, a hermitian metric and
a natural metric connection $\nabla^\Sigma$. The Dirac operator $D_g$ acts on spinors
$$D_g:C^\infty(\Sigma M)\longrightarrow C^\infty(\Sigma M)$$
defined as the composition $\text{Cliff} \circ \nabla^\Sigma$ in the following way
$$\nabla^\Sigma:C^\infty(\Sigma M)\longrightarrow C^\infty(T^*M\otimes\Sigma M) ,$$
$$\text{Cliff}:C^\infty(TM\otimes\Sigma M)\longrightarrow C^\infty(\Sigma M), $$
where $T^*M \simeq TM$ have been identified by means of the metric. We also have a conformal invariance that in our situation, $\tilde g= u^4g$, reads as follows
$$D_{\tilde g}\psi=u^{-4}D_g(u^2\psi) .$$
The functional space that we are going to define is the Sobolev space $H^{\frac{1}{2}}(\Sigma M)$. First we recall that the Dirac operator $D_g$ on a compact manifold is essentially self-adjoint in $L^2(\Sigma M)$, has compact resolvent and there exists a complete $L^2$-orthonormal basis of eigenspinors $\{\psi_i\}_{i\in\mathbb{Z}}$ of the operator
$$D_g\psi_i=\lambda_i \psi_i ,$$
and the eigenvalues $\{\lambda_i\}_{i\in\mathbb{Z}}$ are unbounded, that is $|\lambda_i|\rightarrow\infty$, as $|i|\rightarrow\infty$.
Now if $\psi\in L^{2}(\Sigma M)$, it has a representation in this basis, namely:
$$\psi=\sum_{i\in \mathbb{Z}}a_{i}\psi_{i}.$$
Let us define the unbounded operator $|D_g|^{s}: L^{2}(\Sigma M)\rightarrow L^{2}(\Sigma M)$ by
$$|D_g|^{s}(\psi)=\sum_{i\in \mathbb{Z}} a_{i}|\lambda_{i}|^{s}\psi_{i}.$$
We denote by $H^s(\Sigma M)$ the domain of $|D_g|^{s}$, namely $\psi\in H^s(\Sigma M)$ if and only if
$$\sum_{i\in \mathbb{Z}} a_{i}^2|\lambda_{i}|^{2s}<+\infty .$$
$H^s(\Sigma M)$ coincides with the usual Sobolev space $W^{s,2}(\Sigma M)$ and for $s <0$, $H^s(\Sigma M)$ is defined as the dual of $H^{-s}(\Sigma M)$. For $s >0$, we can define the inner product
$$\langle u,v\rangle_{s}=\langle|D_g|^{s}u,|D_g|^{s}v\rangle_{L^{2}},$$
which induces an equivalent norm in $H^{s}(\Sigma M)$; we will take
$$\langle u,u\rangle:=\langle u,u\rangle_{\frac{1}{2}}=\|u\|^{2}$$
as our standard norm for the space $H^{\frac{1}{2}}(\Sigma M)$. Even in this case, the Sobolev embedding theorems say that there is a continuous embedding
$$H^{s}(\Sigma M) \hookrightarrow L^p(\Sigma M), \quad 1\leq p \leq 3 ,$$
which is compact if $1\leq p <3$.
Finally, we will decompose $H^{\frac{1}{2}}(\Sigma M)$ in a natural way. Let us consider the $L^2$-orthonormal basis of eigenspinors $\{\psi_i\}_{i\in\mathbb{Z}}$: we denote by $\psi_i^-$ the eigenspinors with negative eigenvalue, $\psi_i^+$ the eigenspinors with positive eigenvalue and $\psi_i^0$ the eigenspinors with zero eigenvalue; we also recall that the kernel of $D_g$ is finite dimensional. Now we set:
$$H^{\frac{1}{2},-}:=\overline{\text{span}\{\psi_i^-\}_{i\in\mathbb{Z}}},\quad
H^{\frac{1}{2},0}:=\text{span}\{\psi_i^0\}_{i\in\mathbb{Z}}, \quad
H^{\frac{1}{2},+}:=\overline{\text{span}\{\psi_i^+\}_{i\in\mathbb{Z}}},$$
where the closure is taken with respect to the $H^{\frac{1}{2}}$-topology. Therefore we have the orthogonal decomposition  $H^{\frac{1}{2}}(\Sigma M)$ as:
$$H^{\frac{1}{2}}(\Sigma M)=H^{\frac{1}{2},-}\oplus H^{\frac{1}{2},0}\oplus H^{\frac{1}{2},+}.$$
Also, we let $P^{+}$ and $P^{-}$ be the projectors on $H^{\frac{1}{2},+}$ and $H^{\frac{1}{2},-}$ respectively.\\
Finally, sometimes we will denote by $\mathcal{H}$ the product space $\mathcal{H}=H^1(M) \times H^{\frac{1}{2}}(\Sigma M)$.


\section{Regularity}

\noindent
Here we will prove the regularity of weak solutions of the system of equations (\ref{el}). Due to the critical nonlinearity, the bootstrap argument does not work, and we explicitly note the two equations in (\ref{el}) are strongly coupled, therefore we cannot apply the existing results for the conformal Yamabe equation and the conformal Dirac one separately; anyway we will proceed as in \cite{I} and we will be able to prove the following
\begin{theorem}
Let $(u,\psi)\in H^1(M) \times H^{\frac{1}{2}}(\Sigma M)$ be a weak solution of the system of equations (\ref{el}), then $(u,\psi)\in C^{2,\alpha}(M) \times C^{1,\beta}(\Sigma M)$, for some $0<\alpha,\beta<1$.
\end{theorem}
\begin{proof}
First of all, by the Sobolev embedding there is a continuous injection
$$H^1(M)\hookrightarrow L^p(M), \quad 1\leq p \leq 6 ,$$
$$H^{\frac{1}{2}}(\Sigma M) \hookrightarrow L^p(\Sigma M), \quad 1\leq p \leq 3.$$
Now let $\rho,\eta \in C^{\infty}(M)$, with $\eta=1$ on $supp(\rho)$ and let us denote $B=supp(\eta)$. We compute
\begin{align}
L_g(\rho u)&=-\Delta_g(\rho u)+\frac{1}{8}R_g \rho u  \notag\\
&=-\rho\Delta_g u-u\Delta_g \rho-2g(\nabla\rho,\nabla u)+\frac{1}{8}R_g \rho u \notag\\
&=\rho L_g( u)-u\Delta_g \rho-2g(\nabla\rho,\nabla u) \notag\\
&=\eta|\psi|^2 \rho u -u\Delta_g \rho-2g(\nabla\rho,\nabla u) \notag ,
\end{align}
and
$$D_g(\rho \psi)=\rho D_g \psi+\nabla\rho\cdot\psi=\eta u^2 \rho\psi+\nabla\rho\cdot\psi,$$
where we have denoted by $\cdot$ the Clifford multiplication for brevity. Now, since
$$u\in H^1(M),\qquad\psi \in L^3(\Sigma M),$$
we have that also
$$u\Delta_g \rho+2g(\nabla\rho,\nabla u)\in L^2(M),\qquad \nabla\rho\cdot\psi \in L^3(\Sigma M).$$
We define the two maps:
$$P_1:W^{2,q}(M)\longrightarrow L^{q}(M), \qquad P_1(v)=\eta|\psi|^2v,$$
$$P_2:W^{1,p}(\Sigma M)\longrightarrow L^{p}(\Sigma M), \qquad P_2(\phi)=\eta u^2 \phi.$$
By H\"{o}lder's inequality, the previous maps are well defined if $1<q<\frac{3}{2}$ and $1<p<3$, moreover there are constants depending on $q$ and $p$, such that the operator norms are bounded as follows:
$$\|P_1\|_{op}\leq C_q \|\psi\|^2_{L^3(\Sigma B)}, \qquad \|P_2\|_{op}\leq C_p \|u\|^2_{L^6(B)} .$$
In this way the operators
$$L_g-\eta|\psi|^2:W^{2,q}(M)\longrightarrow L^{q}(M), \qquad 1<q<\frac{3}{2} ,$$
$$D_g-\eta u^2:W^{1,p}(\Sigma M)\longrightarrow L^{p}(\Sigma M), \qquad 1<p<3 ,$$
are invertible if $\|\psi\|_{L^3(\Sigma B)}$ and $\|u\|_{L^6(B)}$ are small, which is possible by taking $B$ even smaller. Therefore there are unique solutions $v\in W^{2,q}(M)$ and $\phi\in W^{1,p}(\Sigma M)$ to the equations
\begin{align}
L_g v-\eta|\psi|^2v & = -u\Delta_g \rho-2g(\nabla\rho,\nabla u) \notag ,\\
D_g \phi-\eta u^2 \phi & =\nabla\rho\cdot\psi \notag,
\end{align}
if $1<q<\frac{3}{2}$ and $1<p<3$.\\
Now we will consider the two dual maps, defined as follows:
$$\tilde P_1:L^{6}(M)\longrightarrow W^{-1,6}(M), \qquad \tilde P_1(\tilde v)=\eta|\psi|^2 \tilde v,$$
$$\tilde P_2:L^{3}(\Sigma M)\longrightarrow W^{-1,3}(\Sigma M), \qquad \tilde P_2(\tilde \phi)=\eta u^2 \tilde \phi.$$
Again, by H\"{o}lder's inequality and Sobolev embedding, the previous maps are well defined and there exist constants, such that the operator norms are bounded as follows:
$$\|\tilde P_1\|_{op}\leq C_q \|\psi\|^2_{L^3(\Sigma B)}, \qquad \|P_2\|_{op}\leq C_p \|u\|^2_{L^6(B)} .$$
Even in this case, the operators
$$L_g-\eta|\psi|^2:L^{6}(M)\longrightarrow W^{-1,6}(M) ,$$
$$D_g-\eta u^2:L^{3}(\Sigma M)\longrightarrow W^{-1,3}(\Sigma M) ,$$
are invertible if $\|\psi\|_{L^3(\Sigma B)}$ and $\|u\|_{L^6(B)}$ are small; therefore there are unique solutions $\tilde v\in L^6(M)$ and $\tilde \phi\in L^3(\Sigma M)$ to the equations
\begin{align}
L_g \tilde v-\eta|\psi|^2\tilde v & =-u\Delta_g \rho-2g(\nabla\rho,\nabla u) \notag ,\\
D_g \tilde \phi-\eta u^2 \tilde \phi & =\nabla\rho\cdot\psi  \notag .
\end{align}
Moreover, since
$$ W^{2,q}(M) \hookrightarrow L^6(M), \quad \frac{6}{5}\leq q <\frac{3}{2} ,$$
$$W^{1,p}(\Sigma M) \hookrightarrow L^p(\Sigma M), \quad \frac{3}{2}\leq p < 3 ,$$
then by the uniqueness $\tilde v=v=\rho u$ and $\tilde \phi=\phi=\rho \psi$, under the above conditions on $q$ and $p$. Now, since $\rho$ and $\eta$ are smooth functions with arbitrary small supports, we have that $u\in  W^{2,q}(M)$ and $\psi \in W^{1,p}(\Sigma M)$, provided $\frac{6}{5}\leq q <\frac{3}{2}$ and $\frac{3}{2}\leq p < 3$. Therefore, by the Sobolev embedding, we get that $u\in L^q(M)$ and $\psi \in L^p(\Sigma M)$, for any $1< q,p <\infty$; and then by plugging them in the initial equations, we have that $u\in  W^{2,q}(M)$ and $\psi \in W^{1,p}(\Sigma M)$, for any $1< q,p <\infty$, by the elliptic regularity estimates. Once more, by the Sobolev embedding for the H\"{o}lder spaces, we have that there exist $0<\alpha,\beta<1$ such that $u\in C^{0,\alpha}(M)$ and $\psi \in C^{0,\beta}(\Sigma M)$; finally by the elliptic regularity again, we get $u\in C^{2,\alpha}(M)$ and $\psi \in C^{1,\beta}(\Sigma M)$.
\end{proof}


\section{Classification of the (PS) sequences}

\noindent
Here we will prove Theorem (\ref{first}). We will need many preliminary propositions and lemmata.
\begin{proposition}
If $\ker D_{g}=\{0\}$, then every (PS) sequence for $E$ is bounded.
\end{proposition}
\begin{proof}
Let $(u_{n},\psi_{n})_{n\in \N} \in H^{1}(M)\times H^{\frac{1}{2}}(\Sigma M)$ be a (PS) sequence for $E$, that is
$$E(u_{n},\psi_{n})\to c, \qquad dE(u_{n},\psi_{n})\to 0,  \text{ in }  H^{-1}(M)\times H^{-\frac{1}{2}}(\Sigma M) .$$
Therefore, there exists a sequence $(\varepsilon_n, \delta_n) \in  H^{-1}(M)\times H^{-\frac{1}{2}}(\Sigma M)$ such that
\begin{equation}\label{equ}
L_{g} u_{n}=|\psi_{n}|^{2}u_{n}+\varepsilon_{n} ,
\end{equation}
\begin{equation}\label{eqp}
D_g\psi_{n}=|u_{n}|^{2}\psi_{n}+\delta_{n} ,
\end{equation}
with
$$\varepsilon_{n}\to 0,  \text{ in } H^{-1}(M) \quad \text{ and } \quad  \delta_{n}\to 0,  \text{ in } H^{-\frac{1}{2}}(M) .$$
We let $z_{n}=(u_{n},\psi_{n}) \in \mathcal{H}$, then
$$2E(z_{n})-\langle dE(z_{n}),z_{n} \rangle = \int_{M}|u_{n}|^{2}|\psi_{n}|^{2}dv_{g} .$$
Hence
\begin{equation}\label{lev}
\int_{M}|u_{n}|^{2}|\psi_{n}|^{2}dv_{g}=2c+o(\|z_{n}\|).
\end{equation}
Multiplying (\ref{equ}) by $u_{n}$ and integrating we have
$$\|u_{n}\|^{2}=\int_{M}|u_{n}|^{2}|\psi_{n}|^{2}dv_{g}+o(\|u_{n}\|),$$
hence
\begin{equation}
\|u_{n}\|^{2}=2c+o(\|z_{n}\|) .
\end{equation}
Now multiplying (\ref{eqp}) by $\psi_{n}^{+}=P^{+}(\psi_n)$, we find
\begin{align}
\|\psi_{n}^{+}\|^{2}&\leq C\int_{M}|u_{n}|^{2}|\psi_{n}||\psi_{n}^{+}|dv_{g} +o(\|\psi_{n}^{+}\|)\notag \\
&\leq C\left(\int_{M}|u_{n}|^{2}|\psi_{n}|^{2}dv_{g}\right)^{\frac{1}{2}} \left(\int_{M}|u_{n}|^{2}|\psi^{+}_{n}|^{2}dv_{g}\right)^{\frac{1}{2}}+o(\|\psi^{+}_{n}\|)\notag \\
&\leq C\big(2c+o(\|z_{n}\|)\big)^{\frac{1}{2}}\|u_{n}\|_{L^{6}}\|\psi_{n}^{+}\|_{L^{3}}+o(\|\psi_{n}\|)\notag \\
&\leq C\big(2c+o(\|z_{n}\|)\big)\|\psi_{n}\| +o(\|\psi_{n}\|)\notag.
\end{align}
Similarly, we have for $\psi_{n}^{-}=P^{-}(\psi_n)$ that
$$\|\psi_{n}^{-}\|^{2}\leq C\big(2c+o(\|z_{n}\|)\big)\|\psi_{n}\| +o(\|\psi_{n}\|).$$
Hence,
$$\|\psi_{n}\|\leq C\big(2c+o(\|z_{n}\|)\big)+o(1),$$
so that
$$\|z_{n}\|\leq C+o(\|z_{n}\|)$$
and $\|z_{n}\|$ is bounded.
\end{proof}

\noindent
From the previous proposition, we have that up to a subsequence, $z_{n}\rightharpoonup z_{\infty}=(u_{\infty},\psi_{\infty})$ in $\mathcal{H}$, also $u_{n}\to u_{\infty}$ in $L^{p}$ for $p<6$ and $\psi_{n}\to\psi$ in $L^{q}$ for $q<3$.
We claim that $z_{\infty}$ is a week solution to (\ref{el}). Indeed, let $z_0=(u_{0},\psi_{0})\in \mathcal{H}$, since $(z_{n})$ is a (PS) sequence for $E$, we have
$$\int_{M}u_{0} L_{g}u_{n}  dv_{g}=\int_{M}|\psi_{n}|^{2}u_{n}u_{0}dv_{g} +o(1),$$
but $|\psi_{n}|^{2}\in L^{\frac{3}{2}}$ and $u_{n}\in L^{6}$, thus $u_{n}|\psi_{n}|^{2}$ converges weakly to $u_{\infty}|\psi_{\infty}|^{2}$ in $L^{\frac{6}{5}}$, hence
$$\int_{M}|\psi_{n}|^{2}u_{n}u_{0}dv_{g}\to \int_{M}|\psi_{\infty}|^{2}u_{\infty}u_{0}dv_{g} .$$
Also, by the weak convergence we have that
$$\int_{M}  u_{0}L_{g}u_{n}  dv_{g}\to \int_{M} u_{0}L_{g}u_{\infty}  dv_{g} .$$
Therefore,
$$L_g u_{\infty}=|\psi_{\infty}|^{2}u_{\infty},$$
and similarly it holds
$$D_{g}\psi_{\infty}=|u_{\infty}|^{2}\psi_{\infty}.$$
We let now $v_{n}=u_{n}-u_{\infty}$ and $\phi_{n}=\psi_{n}-\psi_{\infty}$, then we have the following
\begin{lemma}
Let $h_{n}=(v_{n},\phi_{n})$, then
$$E(h_{n})=E(z_{n})-E(z_{\infty})+o(1)$$
and
$$dE(h_{n})\to 0, \text{ in } H^{-1}(M)\times H^{-\frac{1}{2}}(\Sigma M) .$$
\end{lemma}
\begin{proof}
\begin{align}
2E(z_{n})&=\int_{M}(v_{n}+u_{\infty})L_{g}( v_{n}+u_{\infty})dv_{g} +\int_{M}\langle D_g(\phi_{n}+\psi_{\infty}),\phi_{n}+\psi_{\infty}\rangle dv_{g}\notag\\
&\quad -\int_{M}| v_{n}+u_{\infty}|^{2}|\phi_{n}+\psi_{\infty}|^{2}dv_{g}\notag\\
&=2E(h_{n})+2E(z_{\infty})+\langle dE(z_{\infty}),h_{n}\rangle -\int_{M}|v_{n}|^{2}|\psi_{\infty}|^{2}+2|v_{n}|^{2}\langle \phi_{n},\psi_{\infty}\rangle\notag\\
&\quad +|v_{\infty}|^{2}|\phi_{n}|^{2}+2|\phi_{n}|^{2}v_{n}u_{\infty} +2|\psi_{\infty}|^{2}u_{\infty}v_{n}-4v_{n}u_{\infty}\langle \psi_{\infty}, \phi_{n}\rangle dv_{g} \notag .
\end{align}
Now, we first notice that since $dE(z_{\infty})=0$, we focus on the remaining terms. By the regularity result in Theorem (3.1), we have that $u_{\infty}\in C^{2,\alpha}(M)$ and $\psi_{\infty}\in C^{1,\beta}(\Sigma M)$. Since $v_{n}\to 0$ strongly in $L^{2}(M)$ and $\phi_{n}\to 0$ in $L^{2}(\Sigma M)$ we have that
$$ -\int_{M}|v_{n}|^{2}|\psi_{\infty}|^{2}+2|v_{n}|^{2}\langle \phi_{n},\psi_{\infty}\rangle+|v_{\infty}|^{2}|\phi_{n}|^{2}+2|\psi_{\infty}|^{2}u_{\infty}v_{n}-4v_{n}u_{\infty}\langle \psi_{\infty}, \phi_{n}\rangle dv_{g} \to 0 .$$
The last term is $\int_{M}2|\phi_{n}|^{2}v_{n}u_{\infty}dv_{g}$, but we have that $\phi_{n}\to 0$ in $L^{\frac{5}{2}}(\Sigma M)$ and $v_{n}\to 0$ in $L^{5}(M)$ therefore we conclude that
$$E(z_{n})=E(h_{n})+E(z_{\infty})+o(1),$$
and this finishes the energy estimate. Now for the gradient part $dE$, we denote by $d_{u}E$ and $d_{\psi}E$ the scalar and the spinorial components respectively. We have:
$$d_{u}E(h_{n})=d_{u}E(u_{n},\psi_{n})+d_{u}E(u_{\infty},\psi_{\infty})+|\psi_{n}|^{2}u_{\infty}-|\psi_{\infty}|^{2}u_{n}+2\langle \psi_{n},\psi_{\infty}\rangle v_{n} .$$
But again, $d_{u}E(u_{\infty},\psi_{\infty})=0$ and since $\psi_{n}\to \psi_{\infty}$ in $L^{\frac{12}{5}}(\Sigma M)$ and $u_{n}\to u_{\infty}$ in $ L^{\frac{6}{5}}(M)$ we have that
$$|\psi_{n}|^{2}u_{\infty}-|\psi_{\infty}|^{2}u_{n}\to 0$$
in $L^{\frac{6}{5}}(M)$ hence in $H^{-1}(M)$. We also have that $v_{n}\to 0$ in $L^{\frac{12}{5}}(M)$ and $\psi_{n}\to \psi_{\infty}$ in $L^{\frac{12}{5}}(\Sigma M)$, thus
$\langle \psi_{n},\psi_{\infty}\rangle v_{n}\to 0$ in $L^{\frac{6}{5}}(M)$, thus in $H^{-1}(M)$. Therefore,
$$d_{u}E(h_{n})=o(1), \text{ in } H^{-1}(M).$$
We move now to the spinorial part, that is
$$d_{\psi}E(h_{n})=d_{\psi}E(u_{n},\psi_{n})-d_{\psi}E(u_{\infty},\psi_{\infty})+ |u_{n}|^{2}\psi_{\infty}-|u_{\infty}|^{2}\psi_{n}+2u_{n}u_{\infty}\phi_{n} .$$
Again $d_{\psi}E(u_{\infty},\psi_{\infty})=0$ and $u_{n}\to u_{\infty}$ in $L^{\frac{6}{2}}(M)$ and $\psi_{n}\to \psi_{\infty}$ in $L^{\frac{3}{2}}(\Sigma M)$. Moreover we have that $\phi_{n}\to 0$ in $L^{\frac{5}{2}}(\Sigma M)$ and $u_{n}\to u_{\infty}$ in $L^{\frac{15}{4}}(M)$. It follows that
$$d_{\psi}E(h_{n})=o(1), \text{ in } H^{-\frac{1}{2}}(\Sigma M) .$$
\end{proof}

\noindent
So from now on, we will assume that our (PS) sequence $z_{n}=(u_{n},\psi_{n})$, converges weakly to zero in $H^{1}(M)\times H^{\frac{1}{2}}(\Sigma M)$ and strongly in $L^{p}(M)\times L^{q}(\Sigma M)$, for $p<6$ and $q<3$.\\
We assume that $z_{n}$ does not converge to zero in $H^{1}(M)\times H^{\frac{1}{2}}(\Sigma M)$ since otherwise the (PS) condition would be satisfied. Now, let us denote by $B_{r}(x)$ the geodesic ball with center in $x\in M$ and radius $r$, we define the following sets, for a given $\epsilon_0>0$:
$$\Sigma_{1}=\left\{x\in M; \liminf_{r\to 0} \liminf_{n\to\infty}\int_{B_{r}(x)}|u_{n}|^{6}dv_{g}\geq\epsilon_{0}\right\} ,$$
$$\Sigma_{2}=\left\{x\in M; \liminf_{r\to 0} \liminf_{n\to\infty}\int_{B_{r}(x)}|\psi_{n}|^{3}dv_{g}\geq\epsilon_{0}\right\} ,$$
$$\Sigma_{3}=\left\{x\in M; \liminf_{r\to 0} \liminf_{n\to\infty}\int_{B_{r}(x)}|u_{n}|^{2}|\psi_{n}|^{2}dv_{g}\geq \epsilon_{0}\right\} .$$
We have:
\begin{lemma}
There exists $\epsilon_{0}>0$ depending on $M$, such that if $x_{0}\not \in \Sigma_{1}\cap \Sigma_{2} \cap \Sigma_{3}$, then there exists $r>0$ such that $z_{n}\to 0$ in $H^{1}(B_{r}(x_{0}))\times H^{\frac{1}{2}}(\Sigma B_{r}(x_{0}))$.
\end{lemma}
\begin{proof}
We will prove this result by contradiction, by assuming that for every $\epsilon>0$, there exists $x_{0}\not \in \Sigma_{1}\cap \Sigma_{2} \cap \Sigma_{3}$, such that for every $r>0$, $z_{n}\not \to 0$ in $H^{1}(B_{r}(x_{0}))\times H^{\frac{1}{2}}(\Sigma B_{r}(x_{0}))$.\\

\noindent
\emph{Case I}: $x_{0}\not \in \Sigma_{1}$.\\
Given $\epsilon>0$, there exists $r>0$ such that $\int_{B_{4r}(x_{0})}|u_{n}|^{6}dv_{g}<\epsilon$. We first estimate the $\psi$ component. That is, we consider a smooth cut off function $\eta$ supported on $B_{4r}(x_{0})$ and equals to $1$ on $B_{2r}(x_{0})$, then by (\ref{eqp}) we have:
\begin{align}
D_g(\eta\psi_{n})&=\eta D_g\psi_{n}+\nabla \eta \cdot \psi_{n}\notag \\
&=\eta |u_{n}|^{2}\psi_{n}+\nabla \eta \cdot \psi_{n}+\eta\delta_{n}, \notag
\end{align}
where $\|\delta_{n}\|_{H^{-\frac{1}{2}}}\to 0$. Hence
\begin{align}
\|\eta\psi_{n}\|_{H^{\frac{1}{2}}}&\leq C_{1}\|\eta |u_{n}|^{2}\psi_{n}+\nabla \eta \cdot \psi_{n}+\eta\delta_{n}\|_{H^{-\frac{1}{2}}}\notag\\
&\leq C_{2}\left(\|\eta|u_{n}|^{2}\psi_{n}\|_{L^{\frac{3}{2}}}+ \|\psi_{n}\|_{L^{\frac{3}{2}}}+\|\delta_{n}\|_{H^{-\frac{1}{2}}}\right) .\notag
\end{align}
Since $\|\psi_{n}\|_{L^{\frac{3}{2}}}\to 0$, it remains to estimate
\begin{align}
\|\eta |u_{n}|^{2}\psi_{n}\|_{L^{\frac{3}{2}}}&\leq \|u_{n}\|_{L^{6}(B_{4r}(x_{0}))}^{2}\|\eta \psi_{n}\|_{L^{3}} \leq C_{3}\epsilon^{\frac{1}{3}} \|\eta \psi_{n}\|_{H^{\frac{1}{2}}} .\notag
\end{align}
Hence, taking $C_{3}\epsilon<\frac{1}{2}$, we deduce that $\eta \psi_{n}\to 0$ in $H^{\frac{1}{2}}$, yielding
$$\|\eta \psi_{n}\|_{L^{3}}\to 0.$$
Now we estimate the $u$ component. We consider a smooth cut off function $\rho$ supported on $B_{2r}(x_{0})$ and equals to $1$ on $B_{r}(x_{0})$, then by (\ref{equ}) we have
\begin{align}
L_{g} (\rho u_{n})&=\rho L_{g} u_{n}- u_{n}\Delta_{g} \rho-2g(\nabla \rho, \nabla u_{n})\notag \\
&=\rho|\psi_{n}|^{2}u_{n}- u_{n}\Delta_{g} \rho-2\nabla \rho \cdot \nabla u_{n}+\rho \varepsilon_{n}\notag
\end{align}
where $\varepsilon_{n}\to 0$ in $H^{-1}(M)$. From elliptic estimates now we have that
\begin{align}
\|\rho u_{n}\|_{H^{1}}&\leq C\| \rho|\psi_{n}|^{2 }u_{n}- u_{n}\Delta \rho+2g(\nabla \rho , \nabla u_{n})+\rho \varepsilon_{n}\|_{H^{-1}}\notag\\
&\leq C_{1}\left(\|\rho|\psi_{n}|^{2}u_{n}\|_{L^{\frac{6}{5}}}+\|u_{n}\Delta \rho\|_{L^{\frac{6}{5}}}+2\|g(\nabla \rho , \nabla u_{n})\|_{H^{-1}}\right).\notag
\end{align}
First we estimate $\|\rho|\psi_{n}|^{2}u_{n}\|_{L^{\frac{6}{5}}}$:
\begin{align}
\|\rho|\psi_{n}|^{2}u_{n}\|_{L^{\frac{6}{5}}}&\leq \| \eta \psi_{n}\|^{2}_{L^{3}}\|\rho u_{n}\|_{L^{6}} \leq C_{1}\| \eta \psi_{n}\|^{2}_{H^{\frac{1}{2}}}\|\rho u_{n}\|_{H^{1}} . \notag
\end{align}
From the previous estimates, for $n$ big enough we have that $CC_{1}\| \eta \psi_{n}\|^{2}_{H^{\frac{1}{2}}}<\frac{1}{2}$. Thus we have that
$$\|\rho u_{n}\|_{H^{1}}\leq \|u_{n}\Delta \rho\|_{L^{\frac{6}{5}}}+2\|g(\nabla \rho , \nabla u_{n})\|_{H^{-1}} .$$
Now clearly
$$\|u_{n}\Delta \rho\|_{L^{\frac{6}{5}}}\leq C\|u_{n}\|_{L^{2}},$$
and the term
$$\|g(\nabla \rho , \nabla u_{n})\|_{H^{-1}}\leq \sup_{k\in H^{1};\|k\|_{H^{1}}\leq 1}\left|\int_{M}g(\nabla\rho,\nabla u_{n})kdv_{g}\right| .$$
But
\begin{align}
\left|\int_{M} g(\nabla \rho , \nabla u_{n})kdv_{g}\right|&\leq \left|\int_{M}u_{n}\left(k\Delta \rho+g(\nabla \rho, \nabla k ) \right)dv_{g}\right|\leq C\|k\|_{H^{1}}\|u_{n}\|_{L^{2}} \to 0 .\notag
\end{align}
Hence $\rho u_{n}$ converges to zero in $H^{1}(B_{r}(x_{0}))$ and this leads to a contradiction.\\

\noindent
\emph{Case II}: $x_{0}\not \in \Sigma_{2}$.\\
Given $\epsilon>0$, there exists $r>0$ such that $\int_{B_{4r}(x_{0})}|\psi_{n}|^{3}dv_{g}<\epsilon$. Then again we compute
\begin{align}
L_{g} (\rho u_{n})&=\rho (L_{g} u_{n})- u_{n}\Delta_{g} \rho -2g(\nabla \rho , \nabla u_{n})\notag \\
&=\rho|\psi_{n}|^{2}u_{n}- u_{n}\Delta_{g} \rho-2g(\nabla \rho , \nabla u_{n})+\rho \varepsilon_{n} ,\notag
\end{align}
where $\varepsilon_{n}\to 0$ in $H^{-1}$. From elliptic estimates now we have that
\begin{align}
\|\rho u_{n}\|_{H^{1}}&\leq C\|\rho|\psi_{n}|^{2}u_{n}- u_{n}\Delta_g \rho+2g(\nabla \rho_{1} , \nabla u_{n})+\rho \varepsilon_{n}\|_{H^{-1}}\notag\\
&\leq C_{1}\left(\|\rho|\psi_{n}|^{2}u_{n}\|_{L^{\frac{6}{5}}}+\|u_{n}\Delta_{g} \rho\|_{L^{\frac{6}{5}}}+2\|g(\nabla \rho, \nabla u_{n})\|_{H^{-1}}\right) .\notag
\end{align}
Again, we estimate
\begin{align}
\|\rho|\psi_{n}|^{2}u_{n}\|_{L^{\frac{6}{5}}}&\leq \| \psi_{n}\|^{2}_{L^{3}}\|\rho u_{n}\|_{L^{6}}\leq C_{1}\epsilon^{\frac{2}{3}}\|\rho u_{n}\|_{H^{1}} .\notag
\end{align}
Taking $C_{1}C\epsilon<\frac{1}{2}$, we have that
$$\|\rho u_{n}\|_{H^{1}}\leq C_{1}\left(\|u_{n}\Delta_g \rho\|_{L^{\frac{6}{5}}}+2\|g(\nabla \rho , \nabla u_{n})\|_{H^{-1}} +\|\rho\varepsilon_{n}\|_{H^{-1}}\right),$$
and as in the previous case we have that
$$\|u_{n}\Delta_g \rho\|_{L^{\frac{6}{5}}}+2\|g(\nabla \rho , \nabla u_{n})\|_{H^{-1}} +\|\rho\varepsilon_{n}\|_{H^{-1}}\to 0.$$
Hence $\|\rho u_{n}\|_{H^{1}}\to 0$. Next, we estimate the spinorial component:
\begin{align}
\|\eta\psi_{n}\|_{H^{\frac{1}{2}}}&\leq C_{1}\|\eta |u_{n}|^{2}\psi_{n}+\nabla \eta \cdot \psi_{n}+\eta \delta_{n}\|_{H^{-\frac{1}{2}}}\notag\\
&\leq C_{2}\left(\|\eta |u_{n}|^{2}\psi_{n}\|_{L^{\frac{3}{2}}}+ \|\psi_{n}\|_{L^{\frac{3}{2}}}+\|\delta_{n}\|_{H^{-\frac{1}{2}}}\right).\notag
\end{align}
But
$$\|\eta |u_{n}|^{2}\psi_{n}\|_{L^{\frac{3}{2}}}\leq \|\rho u_{n}\|_{L^{6}}^{2}\|\eta \psi_{n}\|_{H^{\frac{1}{2}}}.$$
Using the fact that $\|\rho u_{n}\|_{H^{1}}\to 0$, we have that $z_{n}\to 0$ in $H^{1}(B_{r}(x_{0}))\times H^{\frac{1}{2}}(\Sigma B_{r}(x_{0}))$, yielding a contradiction.\\

\noindent
\emph{Case III}: $x_{0}\not\in \Sigma_{3}$.\\
Again, given $\epsilon>0$, let $r>0$ so that $\int_{B_{2r}(x_{0})}|u_{n}|^{2}|\psi_{n}|^{2}dv_{g}<\epsilon$. We have that
$$\|\rho \psi_{n}\|_{H^{\frac{1}{2}}} \leq C_{2}\left(\|\rho |u_{n}|^{2}\psi_{n}\|_{L^{\frac{3}{2}}}+o(1)\right).$$
But
$$\|\rho |u_{n}|^{2}\psi_{n}\|_{L^{\frac{3}{2}}}\leq \left(\int_{B_{2r}(x_{0})}|u_{n}|^{2}|\psi_{n}|^{2}dv_{g}\right)^{\frac{1}{2}}\|\rho u_{n}\|_{L^{6}} ,$$
thus
$$\|\rho\psi_{n}\|_{H^{\frac{1}{2}}} \leq C\epsilon^{\frac{1}{2}}\|\rho u_{n}\|_{H^{1}}+o(1).$$
Similarly, we have for the  $u$ component,
$$\|\rho u_{n}\|_{H^{1}}\leq C \|\rho |\psi_{n}|^{2}u_{n}\|_{L^{\frac{6}{5}}}+o(1),$$
and
$$\|\rho |\psi_{n}|^{2}u_{n}\|_{L^{\frac{6}{5}}}\leq \left(\int_{B_{2r}(x_{0})}|u_{n}|^{2}|\psi_{n}|^{2}dv_{g}\right)^{\frac{1}{2}}\|\rho \psi_{n}\|_{L^{3}}. $$
Hence
$$\|\rho u_{n}\|_{H^{1}}\leq C\epsilon^{\frac{1}{2}}\|\rho \psi_{n}\|_{H^{\frac{1}{2}}}+o(1).$$
Combining both the previous inequalities we have $\rho z_{n}\to 0$ in $H^1B_{r}(x_{0})\times H^{\frac{1}{2}}\Sigma B_{r}(x_{0})$, leading to a contradiction.
\end{proof}

\noindent
From the previous lemma we deduce the following properties.
\begin{corollary}
If $(z_{n})$ does not satisfy the (PS) condition, then
$$\Sigma_{1}= \Sigma_{2} =\Sigma_{3}\not=\emptyset.$$
\end{corollary}

\begin{corollary}
Let $(z_{n})$ be a (PS) sequence at the level $c$. If $c<\frac{\epsilon_{0}}{2}$ then $z_{n}$ converges strongly to zero.
\end{corollary}
\begin{proof}
The proof follows from the boundedness of the (PS) sequences. Indeed, from $(\ref{lev})$ we have
$$\int_{M}|u_{n}|^{2}|\psi_{n}|^{2}dv_{g}=2c+o(1).$$
Hence if $2c<\epsilon_0$, we have that for $n$ big enough,
$$\int_{M}|u_{n}|^{2}|\psi_{n}|^{2}dv_{g}<\epsilon_0,$$
thus $z_{n}\to 0$.
\end{proof}

\noindent
Now, for a given (PS) sequence $(z_{n})$, we define the concentration function $Q_{n}$ for $r> 0$ by
$$Q_{n}(r)=\sup_{x\in M} \int_{B_{r}(x)}|u_{n}|^{2}|\psi_{n}|^{2}dv_{g} .$$
We explicitly notice that one can define equivalently the $\sup$ on the integrals relative to $\Sigma_{1}$ and $\Sigma_{2}$. We choose $\epsilon>0$ so that $3\epsilon<\epsilon_{0}$, then if $\Sigma_3\not=0$, we have the existence of $x_{n}\in M$ and $R_{n}\to 0$ such that
$$Q_{n}(R_{n})=\int_{B_{x_{n}}(R_{n})}|u_{n}|^{2}|\psi_{n}|^{2}dv_{g}=\epsilon .$$
Without loss of generality, we can always assume that $x_{n}\to x_{0}$ and $i(M)\geq 3$, where $i(M)$ is the injectivity radius of $M$. Also, we define the map $\rho_{n}(x)=exp_{x_{n}}(R_{n}x)$ for $x\in \R^{3}$ such that $R_{n}|x|<3$; we denote also $\sigma_{n}=\rho_{n}^{-1}$. We let $B_{R}^{0}$ denote the Euclidian ball centered at zero and with radius $R$. That is,
$$B_{R}^{0}=\{x\in \R^{3}; |x|< R\}.$$
We can then consider the metric $g_{n}$ on $B_{R}^{0}$ defined by a suitable rescaled of the pull-back of $g$:
$$g_{n}=R_{n}^{-2}\rho_{n}^{*}g .$$
Clearly, the two Riemannian patches $(B_{R}^{0},g_{n})$ and $(B_{RR_{n}}(x_{n}),g)$ are conformally equivalent for $n$ large enough and $g_{n}\to g_{\R^{3}}$ in $C^{\infty}(B_{R}^{0})$. We consider now the identification map  (see \cite{Bour})
$$(\rho_{n})_{*}:\Sigma_{p}(B_{R}^{0},g_{n})\to \Sigma_{\rho_{n}(p)}(B_{RR_{n}}(x_{n}),g),$$
and we set
$$\rho_{n}^{*}(\varphi)=(\rho_{n})_{*}^{-1}\circ \varphi \circ \rho_{n}.$$
Using these maps, we can define the spinors $\Psi_{n}$ on $\Sigma B_{R}^{0}$ by
$$\Psi_{n}=R_{n}\rho_{n}^{*}\psi_{n} ,$$
and from the conformal change of the Dirac operator, we have that
$$D_{g_{n}}\Psi_{n}=R_{n}^{2}\rho_{n}^{*}D_{g}\psi_{n} .$$
So we get:
$$\int_{B_{R}^{0}}\langle D_{g_{n}}\Psi_{n},\Psi_{n}\rangle dv_{g_{n}}=\int_{B_{RR_{n}}(x_{n})}\langle D_{g}\psi_{n},\psi_{n}\rangle dv_{g},$$
$$\int_{B_{R}^{0}}|\Psi_{n}|^{3}dv_{g_{n}}=\int_{B_{RR_{n}}(x_{n})}|\psi_{n}|^{3}dv_{g}.$$
Now we consider the $u$ component, that is we define
$$U_{n}=R_{n}^{\frac{1}{2}}\rho_{n}^{*}u_{n} ,$$
so that by conformal change of the conformal Laplacian, we have:
$$L_{g_{n}}U_{n}=R_{n}^{\frac{5}{2}}\rho^{*}_{n}L_{g}u_{n} .$$
Hence
$$\int_{B_{R}^{0}}U_{n}L_{g_{n}}U_{n}dv_{g_{n}}=\int_{B_{RR_{n}}(x_{n})}u_{n}L_{g}u_{n}dv_{g},$$

\begin{equation}\label{e1}
\int_{B_{R}^{0}}|U_{n}|^{6}dv_{g_{n}}=\int_{B_{RR_{n}}(x_{n})}|u_{n}|^{6}dv_{g},
\end{equation}
and
$$\int_{B_{R}^{0}}|U_{n}|^{2}|\Psi_{n}|^{2}dv_{g_{n}}=\int_{B_{RR_{n}}(x_{n})}|u_{n}|^{2}|\psi_{n}|^{2}dv_{g}.$$
We have the following:
\begin{lemma}
Let us set
$$F_{n}=L_{g_{n}}U_{n}-|\Psi_{n}|^{2}U_{n}, \qquad H_{n}=D_{g_{n}}\Psi_{n}-|U_{n}|^{2}\Psi_{n} .$$
Then
$$F_{n}\to 0  \text{ in } H^{-1}_{loc}(\R^{3}), \qquad  H_{n}\to 0   \text{ in } H^{-\frac{1}{2}}_{loc}(\Sigma \R^{3}) .$$
\end{lemma}

\noindent
Here the convergence in $H^{-1}_{loc}$ is understood in the sense that for all $R>0$,
$$\sup\left\{\langle F_{n}, F\rangle_{H^{-1},H^{1}}; F\in H^{1}(\R^{3}), \; \text{supp}(F)\subset B_{R}^{0},\; \|F\|_{H^{1}}\leq 1\right\}\to 0,$$
and similarly for $H_{n}$.
\begin{proof}
We first notice that by construction, we have that
$$L_{g_{n}}U_{n}-|\Psi_{n}|^{2}U_{n}=R_{n}^{\frac{5}{2}}\rho_{n}^{*}(L_{g}u_{n}-|\psi_{n}|^{2}u_{n}) .$$
Hence we get
$$F_{n}=R_{n}^{\frac{5}{2}}\rho_{n}^{*}(\varepsilon_{n}) ,$$
and similarly
$$H_{n}=R_{n}^{2}\rho_{n}^{*}(\delta_{n}) .$$
Now we consider $F\in H^{1}(\R^{3})$ such that $supp(F)\subset B_{R}^{0}$ and $\|F\|_{H^{1}}\leq 1$. Since $R_{n}\to 0$, then for $n$ big enough we have that:
\begin{align}
\langle F_{n},F\rangle_{H^{-1},H^{1}} &=\int_{B^{0}_{R_{n}^{-1}}}F_{n}F dv_{g_n}\notag \\
&=\int_{B^{0}_{R_{n}^{-1}}}\rho^{*}_{n}(\varepsilon_{n})R_{n}^{\frac{5}{2}} F dv_{g_{n}}\notag\\
&=\int_{B^{0}_{R_{n}^{-1}}}\rho^{*}_{n}(\varepsilon_{n})R_{n}^{-\frac{1}{2}} F dv_{\rho^{*}_{n}g}\notag\\
&=\int_{B_{1}(x_{n})}\varepsilon_{n}R_{n}^{-\frac{1}{2}}\sigma_{n}^{*}(F) dv_{g} .\notag
\end{align}
But we have that $\|R_{n}^{-\frac{1}{2}}\sigma_{n}^{*}(F)\|_{H^{1}}\leq C$, hence
$$\langle F_{n},F\rangle_{H^{-1},H^{1}}\to 0.$$
A similar estimate holds for $H_{n}$.
\end{proof}

\noindent
Now, let us re recall the spaces
$$D^{1}(\R^{3})=\left\{u\in L^{6}(\R^{3}); |\nabla u| \in L^{2}(\R^{3})\right\} $$
and
$$D^{\frac{1}{2}}(\Sigma\R^{3})=\left\{\psi \in L^{3}(\Sigma\R^{3});|\xi|^{\frac{1}{2}}|\widehat{\psi}|\in L^{2}(\R^{3})\right\},$$
where here $\widehat{\psi}$ is the Fourier transform of $\psi$. We have then the following:

\begin{lemma}
For $\epsilon>0$ small enough, there exist $U_{\infty}\in D^{1}(\R^{3})$ and $\Psi_{\infty}\in D^{\frac{1}{2}}(\Sigma \R^{3})$ such that $U_{n}\to U_{\infty}$ in $H^{1}_{loc}(\R^{3})$ and $\Psi_{n}\to \Psi_{\infty}$ in $H^{\frac{1}{2}}_{loc}(\Sigma\R^{3})$. Moreover they satisfy
\begin{equation}\label{eqR3}
\left\{\begin{array}{ll}
-\Delta_{g_{\R^3}} U_{\infty}=|\Psi_{\infty}|^{2}U_{\infty} \\
& \text{ on } \R^{3} .\\
D_{g_{\R^3}}\Psi_{\infty}=|U_{\infty}|^{2}\Psi_{\infty}
\end{array}
\right.
\end{equation}
\end{lemma}

\begin{proof}
Since the sequence $Z_{n}=(U_n,\Psi_n)$ is bounded in $H^{1}_{loc}\times H^{\frac{1}{2}}_{loc}$, for every $\beta \in C^{\infty}_{0}(\R^{3})$, we have that $\beta Z_{n}$ is bounded in $H^{1}\times H^{\frac{1}{2}}$, hence there exist $U_{\infty}$ and $\Psi_{\infty}$ such that $U_{n}\rightharpoonup U_{\infty}$ in $H^{1}_{loc}$ and $U_{n}\to U_{\infty}$ strongly in $L_{loc}^{p}$ for $p<6$. Similarly $\Psi_{n}\rightharpoonup \Psi_{\infty}$ in $H_{loc}^{\frac{1}{2}}$ and strongly in $L_{loc}^{p}$ for $p<3$. Now we notice that from (\ref{e1}), we have that
$$\int_{B_{R}^{0}}|U_{n}|^{6}dv_{g_{n}}=\int_{B_{RR_{n}}(x_{n})}|u_{n}|^{6}dv_{g}.$$
Hence
\begin{equation}\label{e3}
\limsup_{n\to \infty}\int_{B_{R}^{0}}|U_{n}|^{6}dv_{g_{n}}\leq \sup_{n\geq 1}\int_{M}|u_{n}|^{6}dv_{g}<+\infty ,
\end{equation}
hence $U_{\infty}\in L^{6}(\R^{3})$ and similarly $\Psi_{\infty}\in L^{3}(\Sigma\R^{3})$. Also as in the proof of Proposition (4.1), we see that $(U_{\infty},\Psi_{\infty})$ satisfies equation (\ref{eqR3}); hence
$$\int_{\R^{3}}|\nabla U_{\infty}|^{2}dv_{g} <\infty$$
and $\nabla \Psi_{\infty}\in L^{\frac{3}{2}}(\Sigma\R^3)\subset H^{-\frac{1}{2}}(\Sigma\R^{3})$, which leads to the fact that $U_{\infty}\in D^{1}(\R^{3})$ and $\Psi_{\infty}\in D^{\frac{1}{2}}(\Sigma\R^{3})$. Now, using again Lemma (4.2), we can assume at this stage that $\Psi_{\infty}=0$ and $U_{\infty}=0$ by replacing $\Psi_{n}$ by $\Psi_{n}-\Psi_{\infty}$ and $U_{n}$ by $U_{n}-U_{\infty}$. Now let $x\in \R^{3}$, then by assumption we have that for $n$ big enough,
$$\int_{B_{1}^{0}}|U_{n}|^{2}|\Psi_{n}|^{2}dv_{g_n}\leq \epsilon.$$
Let $\beta \in C^{\infty}_{0}(\R^{3})$, then by elliptic regularity, we have that
\begin{align}
\|\beta^{2} U_{n}\|_{H^{1}}&\leq C\left(\|L_{g_{\R^{3}}}(\beta^{2}U_{n})\|_{H^{-1}}+\|\beta^{2}U_{n}\|_{L^{2}}\right)\\
&\leq C\left(\|L_{g_{n}}(\beta^{2}U_{n})\|_{H^{-1}}+\|(L_{g_{\R^{3}}}-L_{g_{n}})(\beta^{2}U_{n})\|_{H^{-1}}+\|\beta^{2}U_{n}\|_{L^{2}}\right).\notag
\end{align}
Now, we have that $\|\beta^{2}U_{n}\|_{L^{2}}\to 0$, and we want to estimate the term
$$\|(L_{g_{\R^{3}}}-L_{g_{n}})(\beta^{2}U_{n})\|_{H^{-1}}.$$
First, we have that for every $F\in H^{1}$:
$$\langle  (L_{g_{\R^{3}}}-L_{g_{n}})(\beta^{2}U_{n}),F \rangle_{H^{-1},H^{1}}=\langle \beta U_{n},\beta (L_{g_{\R^{3}}}-L_{g_{n}})^{*}F\rangle,$$
where $ (L_{g_{\R^{3}}}-L_{g_{n}})^{*}$ is the adjoint of $L_{g_{\R^{3}}}-L_{g_{n}}$ with respect to the metric $g_{\R^{3}}$. Now since $g_{n}\to g_{\R^{3}}$ in $C^{\infty}$, we have that
$$\|(L_{g_{\R^{3}}}-L_{g_{n}})(\beta \cdot )\|_{H^{2},L^{2}}\to 0,$$
and by duality
$$\|\beta(L_{g_{\R^{3}}}-L_{g_{n}})^{*}\|_{L^{2},H^{-2}}\to 0.$$
Similarly, we have also that
$$\|\beta(L_{g_{\R^{3}}}-L_{g_{n}})^{*}\|_{H^{2},L^{2}}\to 0,$$
therefore, by interpolation, we have that
\begin{equation}\label{e4}
\|\beta(L_{g_{\R^{3}}}-L_{g_{n}})^{*}\|_{H^{1},H^{-1}}\to 0.
\end{equation}
So we have that:
\begin{align}
|\langle  (L_{g_{\R^{3}}}-L_{g_{n}})(\beta^{2}U_{n}),F\rangle_{H^{-1},H^{1}}|&\leq C\|\beta U_{n}\|_{H^{1}}\|\beta (L_{g_{\R^{3}}}-L_{g_{n}})^{*}F\|_{H^{-1}}\notag \\
&\leq C\|\beta(L_{g_{\R^{3}}}-L_{g_{n}})^{*}\|_{H^{1},H^{-1}} \|F\|_{H^{1}},\notag
\end{align}
hence
$$\|(L_{g_{\R^{3}}}-L_{g_{n}})(\beta^{2}U_{n})\|_{H^{-1}}\to 0.$$
It remains to estimate the term $\|L_{g_{n}}(\beta^{2}U_{n})\|_{H^{-1}}$, but we have that
$$\|L_{g_{n}}(\beta^{2}U_{n})\|_{H^{-1}}\leq \|\beta^{2}(|\Psi_{n}|^{2}U_{n}+F_{n})\|_{H^{-1}}+o(1),$$
and from Lemma 4.6, we have that $\beta^{2}F_{n}\to 0$ in $H^{-1}$, therefore, we have that
$$\|\beta^{2} U_{n}\|_{H^{1}}\leq C\| \beta^{2}|\Psi_{n}|^{2}U_{n}\|_{H^{-1}}+o(1).$$
Now, if we take $supp(\beta)\in B_{1}^{0}$, we have that
\begin{align}
\|\beta^{2} U_{n}\|_{H^{1}}&\leq C\| \beta^{2}|\Psi_{n}|^{2}U_{n}\|_{L^{\frac{6}{5}}(B_{1}^{0})}+o(1)\notag \\
&\leq C\left( \int_{B_{1}^{0}}|U_{n}|^{2}|\Psi_{n}|^{2}dv_{g_{n}}\right)^{\frac{1}{2}}\|\beta^{2}\Psi_{n}\|_{L^{3}}+o(1)\notag\\
&\leq C\epsilon^{\frac{1}{2}}\|\beta^{2}\Psi_{n}\|_{L^{3}}+o(1) .\notag
\end{align}
A similar computation can be done to show that
$$\|\beta^{2} \Psi_{n}\|_{H^{\frac{1}{2}}}\leq  C\epsilon^{\frac{1}{2}}\|\beta^{2}U_{n}\|_{L^{6}}+o(1),$$
and combining these last two estimates we have that
$$\|\beta^{2}U_{n}\|_{H^{1}}+\|\beta^{2}\Psi_{n}\|_{H^{\frac{1}{2}}}\to 0.$$
\end{proof}

\noindent
It follows from this lemma in particular, since
$$\int_{B_{1}^{0}}|U_{n}|^{2}|\Psi_{n}|^{2}dv_{g_n}=Q(R_{n})=\epsilon,$$
that also
$$\int_{B_{1}^{0}}|U_{\infty}|^{2}|\Psi_{\infty}|^{2}dv_{g_{\R^{3}}}=\epsilon,$$
and hence $U_{\infty}\not=0$ and $\Psi_{\infty}\not=0$ and they satisfy equation (\ref{eqR3}); by the regularity results proved in the previous section, we have that $U_{\infty}\in C^{2,\alpha}(\R^{3})$ and $\Psi_{\infty}\in C^{1,\beta}(\Sigma\R^{3})$.\\
Now, we assume that $x_{n}\to x$ and  we consider a cut-off function $\beta=1$ on $B_{1}(x)$ and $supp(\beta)\subset B_{2}(x)$, we define then $v_{n}\in C^{2,\alpha}(M)$ and $\phi_{n}\in C^{1,\beta}(\Sigma M)$ by
\begin{equation}\label{e5}
v_{n}=R_{n}^{-\frac{1}{2}}\beta\sigma_{n}^{*}(U_\infty)
\end{equation}
and
\begin{equation}\label{e6}
\phi_{n}=R_{n}^{-1}\beta\sigma_{n}^{*}(\Psi_\infty) .
\end{equation}
We are going to prove the following
\begin{lemma}
Let  $\overline{u}_{n}=u_{n}-v_{n}$ and $\overline{\psi}_{n}=\psi_{n}-\phi_{n}$. Then, up to a subsequence, $\overline{u}_{n}\rightharpoonup 0$ in $H^{1}(M)$ and $\overline{\psi}_{n}\rightharpoonup 0$ in $H^{\frac{1}{2}}(\Sigma M)$.
\end{lemma}
\begin{proof}
We already have that $u_{n}\rightharpoonup 0$ and $\psi_{n}\rightharpoonup 0$, thus to prove the lemma we only need to show the weak convergence for $v_{n}$ and $\phi_{n}$: these sequences are bounded in $H^{1}(M)$ and $H^{1/2}(\Sigma M)$ respectively, then up to subsequences, they converge to some limit. So if we show that the distributional limit is zero, then the limit in the desired space is also zero. So let $f\in C^{\infty}(M)$ and $h\in C^{\infty}(\Sigma M)$. We want to show that
$$\int_{M}v_{n}fdv_{g} \to 0$$
and
$$\int_{M}\langle \phi_{n},h\rangle dv_{g}\to 0.$$
We fix $R>0$, then we have
\begin{align}
\int_{B_{R_{n}R}(x_{n})}v_{n}fdv_{g}&=R_{n}^{-\frac{1}{2}}\int_{B_{R_{n}R}(x_{n})}\beta \sigma_{n}^{*}(U_\infty) fdv_{g}\notag\\
&=R_{n}^{\frac{5}{2}}\int_{B_{R}^{0}}\rho_{n}^{*}(\beta) \rho_{n}^{*}(f)U_\infty dv_{g_{n}} .\notag
\end{align}
Hence
$$\left|\int_{B_{R_{n}R}(x_{n})}v_{n}f dv_{g}\right|\leq CR_{n}^{\frac{5}{2}} \|f\|_{\infty}\int_{B_{R}^{0}}|U_\infty|dv_{g_{\R^{3}}}.$$
Also, for $n$ big enough we have that
\begin{align}
\int_{M\setminus B_{R_{n}R}(x_{n})}v_{n}fdv_{g}&=\int_{B_{3}(x_{n})\setminus B_{R_{n}R}(x_{n})}v_{n}fdv_{g}\notag\\
&=R_{n}^{\frac{5}{2}}\int_{B_{3R_{n}^{-1}}^{0}\setminus B_{R}^{0}}\rho_{n}^{*}(\beta) \rho_{n}^{*}(f)U_\infty dv_{g_{n}} .\notag
\end{align}
Hence, we have
\begin{align}
\left|\int_{M\setminus B_{R_{n}R}(x_{n})}v_{n}fdv_{g}\right|&\leq CR_{n}^{\frac{5}{2}}\|f\|_{\infty}\int_{B_{3R_{n}^{-1}}^{0}\setminus B_{R}^{0}}|U_\infty|dv_{g_{\R^{3}}} \notag \\
&\leq C\|f\|_{\infty}\left(\int_{B_{3R_{n}^{-1}}^{0}\setminus B_{R}^{0}}|U_\infty|^{6}dv_{g_{\R^{3}}}\right)^{\frac{1}{6}}.\notag
\end{align}
Based on these last two inequalities we have that
$$\left|\int_{M}v_{n}fdv_{g}\right|\leq C\|f\|_{\infty}\left(R_{n}^{\frac{5}{2}}\int_{B_{R}^{0}}|U_\infty|dv_{g_{\R^{3}}}+\left(\int_{\R^{3}\setminus B_{R}^{0}}|U_\infty|^{6}dv_{g_{\R^{3}}}\right)^{\frac{1}{6}}\right).$$
Letting $n\to \infty$ and then $R\to \infty$ we get the desired result. A similar inequality holds for $\phi_{n}$ and $h$.
\end{proof}

\noindent
Now we estimate the differential, that is
\begin{lemma}
We have
$$dE(v_{n},\phi_{n})\to 0 \qquad \text{ and }\qquad dE(\overline{u}_{n},\overline{\psi}_{n})\to 0 ,$$
in $H^{-1}(M)\times H^{-\frac{1}{2}}(\Sigma M)$.
\end{lemma}
\begin{proof}
We set
$$f_{n}=L_{g}v_{n}-|\phi_{n}|^{2}v_{n}$$
and
$$h_{n}=D_{g}\phi_{n}-|v_{n}|^{2}\phi_{n} .$$
Let $f\in H^{1}(M)$ and $h\in H^{\frac{1}{2}}(\Sigma M)$, we compute
\begin{align}
\int_{M}f_{n}f dv_{g}&=R_{n}^{-\frac{1}{2}}\left(\int_{M}(-\Delta_{g}\beta)\sigma_{n}^{*}(\U) fdv_{g}+2\int_{M}g(-\nabla \beta,\nabla \sigma_{n}^{*}(\U)) fdv_{g} \right)\notag\\
&\quad +R_{n}^{-\frac{1}{2}}\int_{M}\beta L_{g}(\sigma_{n}^{*}(U_\infty))fdv_{g}-R_{n}^{-\frac{5}{2}}\int_{M}\beta^{3} |\sigma_{n}^{*}(\Psi_\infty)|^{2}\sigma_{n}^{*}(U_\infty) fdv_{g}\notag\\
&=R_{n}^{-\frac{1}{2}}\int_{M}(\Delta_{g}\beta) \sigma_{n}^{*}(\U) fdv_{g}+2R_{n}^{-\frac{1}{2}}\int_{M}g(\nabla \beta,\nabla f) \sigma_{n}^{*}(\U)dv_{g} \notag\\
&\quad + R_{n}^{-\frac{5}{2}}\int_{M}\beta \sigma_{n}^{*}(L_{g_{n}}\U)fdv_{g}-R_{n}^{-\frac{5}{2}}\int_{M}\beta^{3} \sigma_{n}^{*}(|\Ps|^{2}\U)fdv_{g}\notag\\
&=R_{n}^{-\frac{1}{2}}\int_{M}(\Delta_{g}\beta) \sigma_{n}^{*}(\U) fdv_{g}+2R_{n}^{-\frac{1}{2}}\int_{M}g(\nabla \beta,\nabla f) \sigma_{n}^{*}(\U)dv_{g} \notag\\
&\quad + R_{n}^{-\frac{5}{2}}\int_{M}\beta \sigma_{n}^{*}\left((L_{g_{n}}-L_{g_{\R^{3}}})V\right)fdv_{g}+R_{n}^{-\frac{5}{2}}\int_{M}(\beta-\beta^{3}) \sigma_{n}^{*}(|\Ps|^{2}\U) f dv_{g}\notag\\
&=I_{1}+I_{2}+I_{3}+I_{4} .\notag
\end{align}
We estimate first $I_{1}$, so we fix $\gamma \in C^{\infty}(M)$ such that $\gamma=1$ on $supp(\beta)$. Then we have
\begin{align}
I_{1}&=R_{n}^{-\frac{1}{2}}\int_{M}(\Delta_{g}\beta)\sigma_{n}^{*}(\U) \gamma fdv_{g}\notag \\
&=R_{n}^{-\frac{1}{2}}\int_{\R^{3}}\rho^{*}_{n}(\Delta_{g}\beta)\U \rho^{*}_{n}(\gamma f)dv_{\rho^{*}_{n}g}\notag\\
&=R_{n}^{\frac{5}{2}}\int_{\R^{3}}\rho^{*}_{n}(\Delta_{g}\beta)\U \rho^{*}_{n}(\gamma f)dv_{g_{n}}.\notag
\end{align}
Thus
\begin{align}
|I_{1}|&\leq R_{n}^{2}\|\rho^{*}_{n}(\Delta_{g}\beta)\U\|_{L^{\frac{6}{5}}(\R^{3},g_{n})}\|R_{n}^{\frac{1}{2}}\rho^{*}_{n}(\gamma f)\|_{L^{6}(\R^{3},g_{n})}\notag \\
&\leq R_{n}^{2}\|\rho^{*}_{n}(\Delta_{g}\beta)\U\|_{L^{\frac{6}{5}}(\R^{3},g_{n})}\|\gamma f\|_{L^{6}(M)}\notag\\
&\leq R_{n}^{2}\|\rho^{*}_{n}(\Delta_{g}\beta)\U\|_{L^{\frac{6}{5}}(\R^{3},g_{n})}\|f\|_{H^{1}(M)}. \notag
\end{align}
But,
\begin{align}
\int_{\R^{3}}|\rho^{*}_{n}(\Delta_{g}\beta)\U|^{\frac{6}{5}}dv_{g_{n}}&\leq C\int_{R_{n}^{-1}exp^{-1}_{x_{n}}(B_{2}(x))\setminus R_{n}^{-1}exp^{-1}_{x_{n}}(B_{1}(x))}|\U|^{\frac{6}{5}}dv_{g_{\R^{3}}}\notag \\
&\leq C R_{n}^{-\frac{12}{5}} \left(\int_{B_{3R_{n}^{-1}}^{0}\setminus B_{\frac{1}{2}R_{n}^{-1}}^{0}}|\U|^{6}dv_{g_{\R^{3}}}\right)^{\frac{1}{5}}.
\end{align}
Therefore,
$$|I_{1}|\leq C\left(\int_{B_{3R_{n}^{-1}}^{0}\setminus B_{\frac{1}{2}R_{n}^{-1}}^{0}}|\U|^{6}dv_{\R^{3}}\right)^{\frac{1}{6}}\|f\|_{H^{1}(M)},$$
and since $\U\in D^{1}(\R^{3})$, we have that
$$|I_{1}|\leq o(1)\|f\|_{H^{1}}.$$
The estimate for $I_{2}$ is very similar to the one of $I_{1}$, indeed we have
\begin{align}
|I_{2}|&\leq R_{n}^{\frac{5}{2}}\int_{\R^{3}}|\rho_{n}^{*}(\nabla \beta)||\U| |\rho_{n}^{*}(\nabla f)|dv_{g_{n}}\notag\\
&\leq R_{n}^{2}\|\rho_{n}^{*}(\nabla \beta)|\U|\|_{L^{2}(\R^{3},g_{n})} \|R_{n}^{\frac{1}{2}}\rho_{n}^{*}(\nabla f)\|_{L^{2}(\R^{3},g_{n})}\notag\\
&\leq R_{n}^{2}\|\rho_{n}^{*}(\nabla \beta)|\U|\|_{L^{2}(\R^{3},g_{n})}\|f\|_{H^{1}(M)} . \notag
\end{align}
But
\begin{align}
\int_{\R^{3}}|\rho^{*}_{n}(\nabla\beta)|^{2}|\U|^{2}dv_{g_{n}}&\leq C\int_{R_{n}^{-1}exp^{-1}_{x_{n}}(B_{2}(x))\setminus R_{n}^{-1}exp^{-1}_{x_{n}}(B_{1}(x))}|\U|^{2}dv_{\R^{3}}\notag \\
&\leq C R_{n}^{-2} \left(\int_{B_{3R_{n}^{-1}}^{0}\setminus B_{\frac{1}{2}R_{n}^{-1}}^{0}}|\U|^{6}dv_{\R^{3}}\right)^{\frac{1}{3}}, \label{e2}
\end{align}
which allows us to conclude that
$$|I_{2}|\leq o(1)\|f\|_{H^{1}}.$$
We move now to $I_{3}$. We first notice that
$$I_{3}=\int_{M} \left( (L_{g_{n}}-L_{\R^{3}})\U\right)R_{n}^{-\frac{1}{2}}\rho_{n}^{*}(\beta\gamma f)dv_{g_{n}}.$$
Therefore,
$$|I_{3}|\leq C\|\beta\rho_{n}^{*}((L_{g_{n}}-L_{\R^{3}})\U)\|_{L^{\frac{6}{5}}(\R^{3},g_{n})}\|f\|_{H^{1}(M)},$$
but
$$\|\beta\rho_{n}^{*}((L_{g_{n}}-L_{\R^{3}})\U)\|_{L^{\frac{6}{5}}}\leq $$
$$\leq\|\beta\rho_{n}^{*}((L_{g_{n}}-L_{\R^{3}})\U)\|_{L^{\frac{6}{5}}(B_{R}^{0})}+ \|\beta\rho_{n}^{*}((L_{g_{n}}-L_{\R^{3}})\U)\|_{L^{\frac{6}{5}}(\R^{3}\setminus B_{R}^{0})}.$$
Now, since $g_{n}\to g_{\R^{3}}$ in $C^{\infty}(B_{R}^{0})$, and since $-\Delta_{g_{\R^3}} \U=|\Ps|^{2}\U\in L^{\frac{6}{5}}(\R^{3})$, we have as in (\ref{e2}), by letting $n \to \infty$ and then $R \to \infty$ that
$$|I_{3}|\leq o(1)\|f\|_{H^{1}}.$$
It remains now to consider the term $I_{4}$:
$$I_{4}=\int_{\R^{3}}\rho_{n}^{*}(\beta-\beta^{3}) |\Ps|^{2}\U R_{n}^{\frac{1}{2}}\rho_{n}^{*}(\gamma f)dv_{g_{n}} .$$
We have
$$|I_{4}|\leq C\|\rho_{n}^{*}(\beta-\beta^{3}) |\Ps|^{2}\U\|_{L^{\frac{6}{5}}(\R^{3},g_{n})}\|f\|_{H^{1}} ,$$
and
$$\int_{\R^{3}}|\rho_{n}^{*}(\beta-\beta^{3})|^{\frac{6}{5}} |\Ps|^{\frac{12}{5}}|\U|^{\frac{6}{5}}dv_{g_{n}}\leq C\|\Ps\|_{L^{3}}^{\frac{6}{5}}\left(\int_{B_{3R_{n}^{-1}}^{0}\setminus B_{\frac{1}{2}R_{n}^{-1}}^{0}}|\U|^{2}|\Ps|^{2}dv_{g_{\R^{3}}}\right)^{\frac{3}{5}}.$$
Then using the fact that $U_{\infty}\in D^{1}(\R^{3})$ and $\Psi_{\infty}\in D^{\frac{1}{2}}(\Sigma\R^{3})$, we have that
$$|I_{4}|\leq o(1)\|f_{1}\|.$$
Therefore, we have that $f_{n}\to 0$ in $H^{-1}(M)$ and a similar convergence holds for $h_{n}\to 0$ in $H^{-\frac{1}{2}}(\Sigma M)$.\\
Next we move to $(\overline{u}_{n},\overline{\psi}_{n})$ and again we fix $f\in H^{1}(M)$. First we notice that
\begin{align}
d_{u}E(\overline{u}_{n},\overline{\psi}_{n})f&=\int_{M}fL_{g}\overline{u}_{n}dv_{g}-\int_{M}|\overline{\psi}_{n}|^{2}\overline{u}_{n} fdv_{g}\notag\\
&=d_{u}E(u_{n},\psi_{n})f-d_{u}E(v_{n},\phi_{n})f+\int_{M}A_{n}fdv_{g} ,\notag
\end{align}
where
$$A_{n}=|\psi_{n}|^{2}v_{n}-|\phi_{n}|^{2}u_{n}+2\langle \psi_{n},\phi_{n}\rangle (u_{n}-v_{n}).$$
Now, since we already proved that $d_{u}E(u_{n},\psi_{n})\to 0$ and $d_{u}E(v_{n},\phi_{n})\to 0$, it is enough to show that $A_{n}\to 0$ in $H^{-1}$. First we have that
\begin{align}
\int_{M\setminus B_{RR_{n}}(x_{n})}|\psi_{n}|^{\frac{12}{5}}|v_{n}|^{\frac{6}{5}}dv_{g}&\leq \|\psi_{n}\|_{L^{3}}^{\frac{12}{5}}\left(\int_{M\setminus B_{RR_{n}}(x_{n})}|v_{n}|^{6}dv_{g}\right)^{\frac{1}{5}}\notag \\
&\leq C \|\psi_{n}\|_{H^{\frac{1}{2}}}^{\frac{12}{5}}\left(\int_{B_{3R_{n}^{-1}}^{0}\setminus B_{R}^{0}}|\U|^{6}dv_{g_{n}}\right)^{\frac{1}{5}}.\notag
\end{align}
But since $\|\psi_{n}\|_{H^{\frac{1}{2}}}$ is bounded and $dv_{g_{n}}\leq Cdv_{g_{\R^{3}}}$, we have that
$$\int_{B_{3R_{n}^{-1}}^{0}\setminus B_{R}^{0}}|\U|^{6}dv_{g_{n}}\to 0 ,$$
as $R\to \infty$ uniformly on $n$; hence
$$\||\psi_{n}|^{2}v_{n}\|_{L^{\frac{6}{5}}(M\setminus B_{RR_{n}}(x_{n}))}\to 0.$$
Similarly, we have that
\begin{align}
\int_{M\setminus B_{RR_{n}}(x_{n})}|\phi_{n}|^{\frac{12}{5}}|u_{n}|^{\frac{6}{5}}dv_{g}&\leq \|u_{n}\|_{L^{6}}^{\frac{6}{5}}\left(\int_{M\setminus B_{RR_{n}}(x_{n})}|\phi_{n}|^{3}dv_{g}\right)^{\frac{4}{5}}\notag \\
&\leq C \|u_{n}\|_{H^{1}}^{\frac{6}{5}}\left(\int_{B_{3R_{n}^{-1}}^{0}\setminus B_{R}^{0}}|\Ps|^{3}dv_{g_{n}}\right)^{\frac{4}{5}}.\notag
\end{align}
Hence the same conclusion holds when $R\to \infty$. To finish this estimate on the exterior domain, we consider the mixed terms:
\begin{align}
\int_{M\setminus B_{RR_{n}}(x_{n})}|\phi_{n}|^{\frac{6}{5}}|\psi_{n}|^{\frac{6}{5}}|u_{n}|^{\frac{6}{5}}dv_{g}&\leq \|u_{n}\|_{L^{6}}^{\frac{6}{5}}\|\psi_{n}\|_{L^{3}}^{\frac{6}{5}}\left(\int_{M\setminus B_{RR_{n}}(x_{n})}|\phi_{n}|^{3}dv_{g}\right)^{\frac{2}{5}}\notag \\
&\leq C \|\psi_{n}\|_{H^{\frac{1}{2}}}^{\frac{6}{5}} \|u_{n}\|_{H^{1}}^{\frac{6}{5}}\left(\int_{B_{3R_{n}^{-1}}^{0}\setminus B_{R}^{0}}|\Ps|^{3}dv_{g_{n}}\right)^{\frac{2}{5}} .\notag
\end{align}
Now
$$
\int_{M\setminus B_{RR_{n}}(x_{n})}|\phi_{n}|^{\frac{6}{5}}|\psi_{n}|^{\frac{6}{5}}|v_{n}|^{\frac{6}{5}}dv_{g}\leq$$ $$\leq\|\psi_{n}\|_{L^{3}}^{\frac{6}{5}}\left(\int_{M\setminus B_{RR_{n}}(x_{n})}|\phi_{n}|^{3}dv_{g}\right)^{\frac{2}{5}} \left(\int_{M\setminus B_{RR_{n}}(x_{n})}|v_{n}|^{6}dv_{g}\right)^{\frac{1}{5}} $$
$$\leq  C \|\psi_{n}\|_{H^{\frac{1}{2}}}^{\frac{6}{5}} \left(\int_{B_{3R_{n}^{-1}}^{0}\setminus B_{R}^{0}}|\Ps|^{3}dv_{g_{n}}\right)^{\frac{2}{5}}\left(\int_{B_{3R_{n}^{-1}}^{0}\setminus B_{R}^{0}}|\U|^{6}dv_{g_{n}}\right)^{\frac{1}{5}} .$$
We need now an estimate inside the ball $B_{RR_{n}}(x_{n})$, that is
$$\int_{B_{RR_{n}}(x_{n})}|A_{n}|^{\frac{6}{5}}dv_{g}=$$
$$
=\int_{B_{RR_{n}}(x_{n})}||\psi_{n}|^{2}(v_{n}-u_{n})+ (|\psi_{n}|^{2}-|\phi_{n}|^{2})u_{n}+2\langle \psi_{n},\phi_{n}\rangle (u_{n}-v_{n})|^{\frac{6}{5}}dv_{g} \leq $$
$$\int_{B_{R}^{0}}\left(|\Psi_{n}|^{2}|\rho_{n}^{*}(\beta) \U -U_{n}|+(|\Psi_{n}|^{2}-|\rho_{n}^{*}(\beta) \Ps|^{2})|U_{n}|  +2|\Psi_{n}||\rho^{*}_{n}(\beta)\Ps| |U_{n}-\rho^{*}_{n}(\beta) \U|\right)^{\frac{6}{5}}dv_{g_{n}}.$$
Hence for $n$ large, we have that
$$\int_{B_{RR_{n}}(x_{n})}|A_{n}|^{\frac{6}{5}}dv_{g}$$
$$\leq \int_{B_{R}^{0}}\left(|\Psi_{n}|^{2}| \U -U_{n}|+ (|\Psi_{n}|^{2}-|\Ps|^{2})|U_{n}|+2|\Psi_{n}||\Ps||U_{n}- \U|\right)^{\frac{6}{5}}dv_{g_{n}},$$
and since $\Psi_{n}\to \Ps$ in $H^{\frac{1}{2}}_{loc}(\Sigma\R^{3})$ and $U_{n}\to \U$ in $H^{1}_{loc}(\R^{3})$, we have that
$$\int_{B_{RR_{n}}(x_{n})}|A_{n}|^{\frac{6}{5}}dv_{g}\to 0,$$
which finishes the proof for $d_{u}E$. The same computations also hold for $d_{\psi}E$.
\end{proof}

\noindent
Now we estimate the energy, that is:
\begin{lemma}
We have
$$E(\overline{u}_{n},\overline{\psi}_{n})=E(u_{n},\psi_{n})-E_{\R^{3}}(\U,\Ps)+o(1).$$
\end{lemma}
\begin{proof}
We have
$$E(\overline{u}_{n},\overline{\psi}_{n})=$$
$$=\frac{1}{2}\left(\int_{M}(u_{n}-v_{n})L_{g}(u_{n}-v_{n})+\langle D_{g}(\psi_{n}-\phi_{n}),\psi_{n}-\phi_{n}\rangle -|\psi_{n}-\phi_{n}|^{2}|u_{n}-v_{n}|^{2}dv_{g}\right)=$$
$$=E(u_{n},\psi_{n})+E(v_{n},\phi_{n})-dE(v_{n},\phi_{n})(u_{n},\psi_{n})+B_{n},$$
where
$$B_{n}=\frac{1}{2}\left(\int_{M}|u_{n}|^{2}|\phi_{n}|^{2}+|v_{n}|^{2}|\psi_{n}|^{2}-2|u_{n}|^{2}\langle \psi_{n},\phi_{n}\rangle -2|\psi_{n}|^{2}u_{n}v_{n}+4u_{n}v_{n}\langle \psi_{n},\phi_{n}\rangle dv_{g}\right).$$
We first notice that since $(u_{n},\psi_{n})$ is bounded in $H^{1}(M)\times H^{\frac{1}{2}}(\Sigma M)$, from Lemma 4.2, we have that $$dE(v_{n},\phi_{n})(u_{n},\psi_{n})\to 0 .$$
We are going to estimate the five terms of $B_{n}$, one by one.
\begin{align}
\int_{M}|u_{n}|^{2}|\phi_{n}|^{2}&=\int_{M}|u_{n}|^{2}R_{n}^{-2}|\beta|^{2}|\sigma_{n}^{*}\Ps|^{2}dv_{g}\notag\\
&=\int_{\R^{3}}|\rho_{n}^{*}(\beta)|^{2}|U_{n}|^{2}|\Ps|^{2}dv_{g_{n}}\notag\\
&=\int_{\R^{3}}|\U|^{2}|\Ps|^{2}dv_{g_{n}}-\int_{\R^{3}}(|\U|^{2}-|\rho_{n}^{*}(\beta) U_{n}|^{2})|\Ps|^{2}dv_{g_{n}}.\notag
\end{align}
For $n$ large, we have
$$\int_{\R^{3}}(|\U|^{2}-|\rho_{n}^{*}(\beta)U_{n}|^{2})|\Ps|^{2}dv_{g_{n}}=$$
$$= \int_{B_{R}^{0}}(|\U|^{2}-|U_{n}|^{2})|\Ps|^{2}dv_{g_{n}}+\int_{\R^{3}\setminus B_{R}^{0}} (|\U|^{2}-|\rho_{n}^{*}(\beta) U_{n}|^{2})|\Ps|^{2}dv_{g_{n}}.$$
Now since $U_{n}\to \U$ in $H^{1}_{loc}(\R^3)$, we have that
$$\int_{B_{R}^{0}}(|\U|^{2}-|U_{n}|^{2})|\Ps|^{2}dv_{g_{n}} \to 0,$$
as $n\to \infty$. Now
$$\left|\int_{\R^{3}\setminus B_{R}^{0}} (|\U|^{2}-|\rho_{n}^{*}(\beta)U_{n}|^{2})|\Ps|^{2}dv_{g_{n}}\right|\leq$$
$$\leq C \|\Ps\|_{L^{3}(\R^{3}\setminus B_{R}^{0})}^{2}\left(\|\U\|_{L^{6}(\R^{3})}^{2}+\|U_{n}\|_{L^{6}(B_{3R^{-1}_{n}}^{0}\setminus B_{R}^{0})}\right).$$
From (\ref{e3}),  $\|U_{n}\|_{L^{6}(B_{3R^{-1}_{n}}^{0})}$ is bounded, hence we can take $R\to \infty$ uniformly in $n$ to get that
$$\int_{M}|u_{n}|^{2}|\phi_{n}|^{2}dv_{g}=\int_{\R^{3}}|\U|^{2}|\Ps|^{2}dv_{g_{\R^{3}}}+o(1).$$
The next term that we want to estimate is
\begin{align}
\int_{M}|v_{n}|^{2}|\psi_{n}|^{2}dv_{g}&=\int_{\R^{3}}|\U|^{2}|\rho_{n}^{*}(\beta)|^{2}|\Psi_{n}|^{2}dv_{g_{n}}\notag\\
&=\int_{\R^{3}}|\U|^{2}|\Ps|^{2}dv_{g_{n}}+\int_{\R^{3}}|\U|^{2}|\left(\rho_{n}^{*}(\beta)|^{2}|\Psi_{n}|^{2}-|\Ps|^{2}\right)dv_{g_{n}} .\notag
\end{align}
Again by splitting the integral in $B_{R}^{0}$ and $\R^{3}\setminus B_{R}^{0}$, we get that
$$\int_{M}|v_{n}|^{2}|\psi_{n}|^{2}dv_{g}=\int_{\R^{3}}|\U|^{2}|\Ps|^{2}dv_{g_{\R^{3}}}+o(1).$$
Next,
\begin{align}
\int_{M}|u_{n}|^{2}\langle \psi_{n},\phi_{n}\rangle dv_{g}&=\int_{\R^{3}}\rho_{n}^{*}(\beta)|U_{n}|^{2}\langle \Psi_{n},\Ps\rangle dv_{g_{n}}\notag\\
&=\int_{\R^{3}}|\U|^{2}|\Ps|^{2}dv_{g_{n}}+\int_{\R^{3}}\langle \rho_{n}^{*}(\beta)|U_{n}|^{2}\Psi_{n}-|\U|^{2}\Ps,\Ps\rangle dv_{g_{n}} .\notag
\end{align}
Similarly, we have
$$\int_{\R^{3}}\langle \rho_{n}^{*}(\beta)|U_{n}|^{2}\Psi_{n}-|\U|^{2}\Ps,\Ps\rangle dv_{g_{n}}=$$
$$=\int_{B_{R}^{0}}\langle |U_{n}|^{2}\Psi_{n}-|\U|^{2}\Ps,\Ps\rangle dv_{g_{n}}+\int_{\R^{3}\setminus{B_{R}^{0}}}\langle \rho_{n}^{*}(\beta)|U_{n}|^{2}\Psi_{n}-|\U|^{2}\Ps,\Ps\rangle dv_{g_{n}}.$$
Using the convergence of $U_{n}$ and $\Psi_{n}$ in $H^{1}_{loc}(\R^3)$ and $H^{\frac{1}{2}}_{loc}(\Sigma \R^3)$ respectively, we have that
$$\int_{B_{R}^{0}}\langle |U_{n}|^{2}\Psi_{n}-|\U|^{2}\Ps,\Ps\rangle dv_{g_{n}}\to 0 ,$$
as $n\to \infty$. For the  term in $\R^{3}\setminus{B_{R}^{0}}$, we have
$$\left|\int_{\R^{3}\setminus{B_{R}^{0}}}\langle \rho_{n}^{*}(\beta)|U_{n}|^{2}\Psi_{n}-|\U|^{2}\Ps,\Ps\rangle dv_{g_{n}}\right|\leq $$
$$\leq C\|\Ps\|_{L^{3}(\R^{3}\setminus B_{R}^{0})}\left(\|\Ps\|_{L^{3}(\R^{3}\setminus B_{R}^{0})}\|\U\|_{L^{6}(\R^{3}\setminus B_{R}^{0})}^{2}+\|\Psi_{n}\|_{L^{3}(B_{3R_{n}^{-1}}^{0}\setminus B_{R}^{0})}\|U_{n}\|_{L^{6}(B_{3R_{n}^{-1}}^{0}\setminus B_{R}^{0})}^{2}\right).$$
Thus
$$\int_{M}|u_{n}|^{2}\langle \psi_{n},\phi_{n}\rangle dv_{g}=\int_{\R^{3}}|\U|^{2}|\Ps|^{2}dv_{g_{\R_{3}}}+o(1) .$$
The remaining two terms can be handled in the same way. Therefore, so far we have
$$E(\overline{u}_{n},\overline{\psi}_{n})=E(u_{n},\psi_{n})+E(v_{n},\phi_{n})-2E_{\R^{3}}(\U,\Ps)+o(1) .$$
Now we estimate $E(v_{n},\phi_{n})$. First, we have
\begin{align}
\int_{M}|v_{n}|^{2}\phi_{n}|^{2}dv_{g}&=\int_{\R^{3}}|\rho_{n}^{*}(\beta)|^{4}|\U|^{2}|\Ps|^{2}dv_{g_{n}}\notag\\
&=\int_{B_{R}^{0}}|\U|^{2}|\Ps|^{2}dv_{g_{n}}+\int_{\R^{3}\setminus B_{R}^{0}}|\rho_{n}^{*}(\beta)|^{4}|\U|^{2}|\Ps|^{2}dv_{g_{n}}, \notag
\end{align}
but
$$\int_{\R^{3}\setminus B_{R}^{0}}|\rho_{n}^{*}(\beta)|^{4}|\U|^{2}|\Ps|^{2}dv_{g_{n}}\leq C\int_{\R^{3}\setminus B_{R}^{0}}|\U|^{2}|\Ps|^{2}dv_{g_{\R^{3}}}.$$
Hence, if we let $R\to \infty$, we have
$$\int_{M}|v_{n}|^{2}|\phi_{n}|^{2}dv_{g}=\int_{\R^3}|\U|^{2}|\Ps|^{2}dv_{g_{\R^{3}}}+o(1).$$
Now,
\begin{align}
\int_{M}v_{n}L_{g}v_{n}dv_{g}&=\int_{\R^{3}}\rho_{n}^{*}(\beta) \U L_{g_{n}}(\rho_{n}^{*}(\beta)\U)dv_{g_{n}}\notag \\
&=\int_{\R^{3}}|\rho_{n}^{*}(\beta)|^{2} \U L_{g_{n}}\U dv_{g_{n}}+\int_{\R^{3}}\rho_{n}^{*}(\beta)(-\Delta_{g_{n}} \rho_{n}^{*}(\beta))|\U|^{2}dv_{g_{n}}\notag \\
& \quad -2\int_{\R^{3}}\U \rho_{n}^{*}(\beta)g(\nabla \rho_{n}^{*}(\beta),\nabla \U) dv_{g_{n}}. \notag
\end{align}
Now
$$\int_{\R^{3}}|\rho_{n}^{*}(\beta)|^{2} \U L_{g_{n}}\U dv_{g_{n}}=$$
$$=\int_{\R^{3}}|\rho_{n}^{*}(\beta)|^{2}\U L_{g_{\R^{3}}}\U dv_{g_{n}}+\int_{\R^{3}}|\rho_{n}^{*}(\beta)|^{2} \U (L_{g_{n}}-L_{g_{\R^{3}}})\U dv_{g_{n}}.$$
Clearly,
$$\int_{\R^{3}}|\rho_{n}^{*}(\beta)|^{2}\U L_{g_{\R^{3}}}\U dv_{g_{n}}=\int_{\R^{3}}|\nabla \U|^{2}dv_{g_{\R^{3}}}+o(1),$$
and
$$\left|\int_{\R^{3}}|\rho_{n}^{*}(\beta)|^{2} \U (L_{g_{n}}-L_{g_{\R^{3}}})\U dv_{g_{n}}\right|\leq $$
$$ \leq \int_{B_{R}^{0}} |\U| |(L_{g_{n}}-L_{g_{\R^{3}}})\U | dv_{g_{n}}+C\int_{B_{3R_{n}^{-1}}^{0}\setminus B_{R}^{0}}|\U| |(L_{g_{n}}-L_{g_{\R^{3}}})\U| dv_{g_{n}}.$$
The first term converges to zero as $n\to \infty$ from (\ref{e4}). For the second term, we use the fact that $\nabla \U\in L^{2}(\R^{3})$ hence it converges to zero if we let $R\to \infty$ uniformly on $n$. Also
$$\left|\int_{\R^{3}}\rho_{n}^{*}(\beta)(-\Delta_{g_{n}} \rho_{n}^{*}(\beta))|\U|^{2}dv_{g_{n}}\right|\leq C\|\U\|^{2}_{L^{6}(B_{3R_{n}^{-1}}^{0}\setminus B_{\frac{1}{2}R_{n}^{-1}}^{0})}=o(1),$$
and
$$\left|\int_{\R^{3}}\U \rho_{n}^{*}(\beta)g(\nabla \rho_{n}^{*}(\beta),\nabla \U) dv_{g_{n}}\right|\leq $$
$$\leq C\|\nabla \U\|_{L^{2}(B_{3R_{n}^{-1}}^{0}\setminus B_{\frac{1}{2}R_{n}^{-1}}^{0})}\|\U\|_{L^{6}(B_{3R_{n}^{-1}}^{0}\setminus B_{\frac{1}{2}R_{n}^{-1}}^{0})}=o(1).$$
Hence
$$\int_{M}v_{n}L_{g}v_{n}dv_{g}=\int_{\R^{3}}|\nabla \U|^{2}dv_{g_{\R^{3}}}+o(1),$$
and similarly
$$\int_{M}\langle D_{g} \phi_{n},\phi_{n}\rangle dv_{g}=\int_{\R^{3}}\langle D_{g_{\R^{3}}}\Ps,\Ps \rangle dv_{g_{\R^{3}}}+o(1).$$
Combining al these estimates, we have that
$$E(\overline{u}_{n},\overline{\psi}_{n})=E(u_{n},\psi_{n})-E_{\R^{3}}(\U,\Ps)+o(1).$$
\end{proof}

\noindent
Now we will prove the following energy lower bound for solutions in $\R^3$.
\begin{proposition}
Let $(U,\Psi)\in D^{1}(\R^{3})\times D^{\frac{1}{2}}(\Sigma \R^{3})$ be a non trivial solution of
$$\left\{\begin{array}{ll}
-\Delta_{g_{\R^3}} U=|\Psi|^{2}U \\
& \text{ on } \R^{3} .\\
D_{g_{\R^3}}\Psi=|U|^{2}\Psi
\end{array}
\right.
$$
Then
$$E_{\R^{3}}(U,\Psi)\geq \tilde{Y}_{g_0}(S^{3})=: c_{3}.$$
\end{proposition}
\begin{proof}
First we recall that
$$Y_{g_{0}}(S^{3})\leq\frac{\displaystyle\int_{\R^{3}}|\nabla U|^{2}dv_{g_{\R^{3}}}} {\left(\displaystyle\int_{\R^{3}}|U|^{6}dv_{g_{\R^{3}}}\right)^{\frac{1}{3}}}.$$
But
\begin{align}
\int_{\R^{3}}|\nabla U|^{2}dv_{g_{\R^{3}}}&=\int_{\R^{3}}|U|^{2}|\Psi|^{2}dv_{g_{\R^{3}}}\notag\\
&\leq \left(\int_{\R^{3}}|U|^{6}dv_{g_{\R^{3}}}\right)^{\frac{1}{3}} \left(\int_{\R^{3}}|\Psi|^{3}dv_{g_{\R^{3}}}\right)^{\frac{2}{3}} .\notag
\end{align}
Thus
$$Y_{g_{0}}(S^{3})\leq \left(\int_{\R^{3}}|\Psi|^{3}dv_{g_{\R^{3}}}\right)^{\frac{2}{3}}.$$
On the other hand, if we denote by $(u,\psi)$ the pull-back of $(U,\Psi)$ by the standard stereographic projection, we have
$$\lambda_{g_{0}}^{+}(S^{3})\leq \frac{\displaystyle\left(\int_{S^{3}}|D_{g_0}\psi|^{\frac{3}{2}}dv_{g_0}\right)^{\frac{4}{3}}}{\displaystyle\int_{S^{3}}\langle D_{g_0}\psi,\psi\rangle dv_{g_0}}.$$
Again, we have
$${\int_{S^{3}}\langle D_{g_0}\psi,\psi\rangle dv_{g_0}}=\int_{S^{3}}|u|^{2}|\psi|^{2}dv_{g_0},$$
and
\begin{align}
\int_{S^{3}}|D_{g_0}\psi|^{\frac{3}{2}}dv_{g_0}&=\int_{S^{3}}|u|^{3}|\psi|^{\frac{3}{2}}dv_{g_0}\notag\\
&\leq \left(\int_{S^{3}}|u|^{6}dv_{g_0}\right)^{\frac{1}{3}} \left(\int_{S^{3}}|u|^{2}|\psi|^{2}dv_{g_0}\right)^{\frac{3}{4}} .\notag
\end{align}
Hence,
$$\lambda_{g_{0}}^{+}(S^{3})\leq \left(\int_{S^{3}}|u|^{6}dv_{g_0}\right)^{\frac{1}{3}}.$$
Now, using the Sobolev inequality
$$Y_{g_{0}}(S^{3})\left(\int_{\R^{3}}|U|^{6}dv_{g_{\R^{3}}}\right)^{\frac{1}{3}}\leq \int_{\R^{3}}|\nabla U|^{2}dv_{g_{\R^{3}}},$$
hence
$$E_{\R^3}(U,\Psi)=\frac{1}{2}\int_{\R^{3}}|U|^{2}|\Psi|^{2}dv_{g_{\R^{3}}}\geq \frac{1}{2}Y_{g_{0}}(S^{3})\lambda_{g_{0}}^{+}(S^{3})=c_3.$$
\end{proof}

\noindent
\begin{proof}  \emph{of Theorem (\ref{first}})\\
From the previous results, we can re-iterate the process $m$ times, for $(\overline{u}_{n},\overline{\psi}_{n})$ since they satisfy the same assumptions as $(u_{n},\psi_{n})$, and we will have
$$E(u_{n},\psi_{n})=E(u_{\infty},\psi_{\infty})+\sum_{k=1}^{m}E_{\R^{3}}(U_\infty^{k},\Ps^{k})+o(1),$$
where $(U_\infty^{k},\Psi_\infty^{k})$ are solutions to equations (\ref{el}) on $\mathbb{R}^3$. Now using the fact that
$$\sum_{k=1}^{m}E_{\R^{3}}(\U^{k},\Ps^{k})\geq mc_{3} ,$$
we stop the process when $c-mc_{3}<\frac{\epsilon_{0}}{2}$. Then from Corollary 4.5, we have the existence of $m$ sequences $x_{n}^{1},\cdots x_{n}^{m}$ such that $x_{n}^{k}\to x^{k}\in M$, and $m$ sequences of real numbers $R_{n}^{1},\cdots, R_{n}^{m}$ converging to zero, such that
$$u_{n}=u_{\infty}+v_{n}^{1}+\cdots +v_{n}^{m}+o(1) \text{ in } H^{1}(M),$$
$$\psi_{n}=\psi_{\infty}+\phi_{n}^{1}+\cdots +\phi_{n}^{m}+o(1) \text{ in } H^{\frac{1}{2}}(\Sigma M),$$
where
$$v_{n}^{k}=(R_{n}^{k})^{-\frac{1}{2}}\beta_{k}\sigma_{n,k}^{*}(U_\infty^{k}) ,$$
$$\phi_{n}^{k}=(R_{n}^{k})^{-1}\beta_{k}\sigma_{n,k}^{*}(\Psi_\infty^{k}) ,$$
with $\sigma_{n,k}=(\rho_{n,k})^{-1}$ and $\rho_{n,k}(\cdot)=exp_{x_{n}^k}(R_{n}^k \cdot)$.
Also $\beta_{k}$ are smooth compactly supported functions, such that $\beta_{k}=1$ on $B_{1}(x^{k})$ and $supp(\beta_{k})\subset B_{2}(x^{k})$. Moreover, we have
$$E(u_{n},\psi_{n})=E(u_{\infty},\psi_{\infty})+\sum_{k=1}^{p}E_{\R^{3}}(\U^{k},\Ps^{k})+o(1).$$
\end{proof}


\section{Existence of a positive solution}

\noindent
In this section we will prove Theorem (\ref{second}): for our convenience we will split the two statements, and we will prove them separately. First we need the following characterization of the first eigenvalue of the Dirac operator: let us fix $u>0$ and let us consider the minimization problem
$$\tilde{\lambda}_u=\inf\left\{I(\psi); \text{ for  $\psi\in H^{\frac{1}{2}}(\Sigma M)$ s.t. }  I(\psi)> 0, \; P^{-}\left(D_g \psi-I(\psi)u^{2}\psi\right) =0 \right\} ,$$
then we have

\begin{proposition}
For a given $u>0$ and smooth, we have that $\tilde{\lambda}_u>0$. Moreover, the minimization problem is achieved and $\tilde{\lambda}_u$ is the first eigenvalue for the Dirac operator $D_{g_{u}}$, where $g_u=u^4g$.
\end{proposition}
\begin{proof}
Let $\psi_{n}$ be a minimizing sequence, that is $I(\psi_{n})\to \tilde{\lambda}_u$.  Without loss of generality, we can assume that $\int_{M}u^{2}|\psi_n|^{2}dv_{g}=1$. Then we have
$$I(\psi_{n})=\|\psi_{n}^{+}\|^{2}-\|\psi_{n}^{-}\|^{2},$$
but
$$-\|\psi_{n}^{-}\|^{2}=I(\psi_{n})\int_{M}u^{2}\langle \psi_n,\psi_n^{-}\rangle dv_{g},$$
hence, using Holder's inequality, we have
$$\|\psi_{n}^{-}\|^{2}\leq I(\psi_{n})\left(\int_{M}u^{2}|\psi_n^{-}|^{2}dv_{g}\right)^{\frac{1}{2}}.$$
Since, the projector $P^{-}:H^{\frac{1}{2}}(\Sigma M)\to H^{\frac{1}{2},-}$ is a pseudo-differential operator of order zero, we have that
$$\|\psi_{n}^{-}\|^{2}\leq C I(\psi_{n}),$$
thus, we have
$$\|\psi_{n}^{+}\|^{2}\leq C I(\psi_{n}).$$
Therefore, if $I(\psi_{n})\to 0$, we would have that $\psi_{n}\to 0$ in $H^{\frac{1}{2}}(\Sigma M)$, contradicting the fact that $\int_{M}u^{2}|\psi_n|^{2}dv_{g}=1$. Therefore $\tilde{\lambda}_u>0$, and any minimizing sequence has $\|\psi_{n}\|\geq \delta>0$.\\
\noindent
Now we will prove the existence of a minimizer.  We notice that without the condition
$$\int_{M}\langle D_g\psi-I(\psi)u^{2}\psi,\varphi\rangle dv_{g}=0, \quad \forall \varphi\in H^{\frac{1}{2},-},$$
we would be in a classical variational setting allowing us to find a minimizer: unfortunately this is not the case, so we use here the idea in \cite{Pan}, later on inproved in \cite{Sul}. First, we claim that
$$S=\{\psi\in H^{\frac{1}{2}}(\Sigma M);P^{-}(D\psi-I(\psi)u^{2}\psi)=0\} , $$
is a manifold. Indeed, we consider the operator
$$G:H^{\frac{1}{2}}(\Sigma M)\to H^{\frac{1}{2},-}, \qquad G(\psi)=P^{-}(D_g\psi-I(\psi)u^{2}\psi),$$
so that $S=G^{-1}(0)$; therefore, if $DG$ is onto, then $S$ will be indeed a manifold. We compute then $DG$:
$$DG(\psi)h=P^{-1}(D_gh-I(\psi)u^{2}h).$$
If we restrict this operator to $H^{\frac{1}{2},-}$, we have that
$$\langle DG(\psi)h,h\rangle=-\|h\|^{2}-I(\psi)\int_{M}u^{2}|h|^{2}dv_{g} .$$
Thus, $DG(\psi)$ is definite negative and hence invertible thus onto and $S$ is a manifold. Now, using Ekeland's variational principle, see \cite{Ek}, we can find a minimizing sequence for $\tilde{\lambda}_u$ that is a (PS) sequence. We want to show that it is still a (PS) sequence in $H^{\frac{1}{2}}(\Sigma M)$. So let $\psi_{n}$ be a (PS) sequence for $I$ in $S$. We write $DI(\psi_{n})=\varepsilon_{n}$. We have that $\varepsilon_{n}^{T}$ (the tangential component of $\varepsilon_{n}$ on the tangent space of the manifold $S$) already converges to zero since $\psi_{n}$ is a (PS) sequence for $I$ on $S$. We want to show that also the normal component converges to zero.\\
As we did previously, the operator $DG(\psi_{n}):H^{\frac{1}{2}}(\Sigma M)\to H^{\frac{1}{2},-}$ is onto, hence it has a left inverse $K_{n}:H^{\frac{1}{2},-} \to H^{\frac{1}{2},-}$. Moreover, since we can always assume that $\int_{M}u^{2}|\psi_{n}|^{2}dv_{g}=1$, we have that $\|K_{n}\|_{op}\leq C$. The operator $P_{n}=A_{n}\circ DG(\psi_{n})$ is now a projector on $H^{\frac{1}{2},-}$ parallel to $T_{\psi_{n}}S$. Indeed, we have that if $h\in H^{\frac{1}{2},-}$ then $P_{n}h=h$ and $T_{\psi_{n}}S=\ker DG(\psi_{n})$. We consider then the adjoint of $P_{n}$, denoted by $P_{n}^{*}$: it is a projector of $N_{\psi}S$  (the orthogonal space to the tangent space) parallel to $H^{\frac{1}{2},+}$. Now we notice that $\langle \varepsilon_{n}, \varphi \rangle =0$ for all $\varphi\in H^{\frac{1}{2},-}$, hence $\varepsilon_{n}\in H^{\frac{1}{2},+}$ so we have
$$\varepsilon_{n}=P_{n}^{*}\varepsilon_{n}^{T}.$$
Thus $\varepsilon_{n}\to 0$ and $\psi_{n}$ is a (PS) sequence for $I$ in $H^{\frac{1}{2}}(\Sigma M)$.
This (PS) sequence then satisfies
\begin{align}
\|\psi_{n}^{+}\|^{2}&\leq \int_{M}u^{2}|\psi^{+}_{n}||\psi_{n}|dv_{g}+o(\|\psi_{n}^{+}\|)\leq C +o(\|\psi_{n}^{+}\|),\notag
\end{align}
and a similar inequality holds for $\psi_{n}^{-}$ which gives us the boundedness in $H^{\frac{1}{2}}(\Sigma M)$. The rest of the proof is classical in order to show that $\psi_{n}\to \psi$ in $H^{\frac{1}{2}}(\Sigma M)$ and $\psi$ satisfies
$$D_g\psi=\tilde{\lambda}_uu^{2}\psi.$$
Finally, since $\tilde{\lambda}_u$ is a minimizer then by a conformal change of the metric $g_u= u^{4}g$ we have that $$\tilde{\lambda}_u=\lambda_{1}(D_{g_u}).$$
\end{proof}

\noindent
Now we prove the following:
\begin{proposition}
It holds:
$$Y_g(M)\lambda^{+}_g(M)\leq \tilde{Y}_g(M) \leq Y_{g_0}(S^{3})\lambda^{+}_{g_0}(S^{3})=\tilde{Y}_{g_0}(S^{3}) .$$
\end{proposition}
\begin{proof}
We first notice that using the Sobolev inequality and Proposition (5.1), we have for any non-trivial $(u,\psi)\in H^{1}(M) \times H^{\frac{1}{2}}(\Sigma M)$:
\begin{align}
\tilde{E}(u,\psi)&\geq Y_g(M)\left(\int_{M}u^{6}dv_{g}\right)^{\frac{1}{3}}\frac{\displaystyle\int_{M}\langle D_g \psi,\psi\rangle dv_{g}}{\displaystyle\int_{M}|u|^{2}|\psi|^{2}dv_{g}}\notag\\
&\geq Y_g(M)\lambda_{1}(g_{u})Vol(g_{u})^{\frac{1}{3}}\notag \\
&\geq Y_g(M)\lambda_g^{+}(M),
\end{align}
which proves the first inequality of the claim. We also recall that we set $g_u=u^4g$. Now for the second inequality, we consider a spinor $\overline{\Psi}_{0}\in (\Sigma\R^{3})$ such that $|\overline{\Psi}_{0}|=\frac{1}{\sqrt{2}}$ and we define the standard spinor
$$\Psi_{0}(x)=\left(\frac{2}{1+|x|^{2}}\right)^{\frac{3}{2}}(1-x)\cdot \overline{\Psi}_{0},$$
and bubble
$$U_{0}=\left(\frac{2}{1+|x|^{2}}\right)^{\frac{1}{2}}.$$
Then an easy computation shows that $(U_{0},\Psi_{0})$ is a critical point for $E_{\R^{3}}$ and
$$E(U_{0},\Psi_{0})=\frac{1}{2}Y_{g_{0}}(S^{3})\lambda_{g_{0}}^{+}(S^{3}).$$
Now, we fix $x_{0}\in M$ and let $\lambda>0$ be a parameter that we will be tending to zero. We define as in (\ref{e5})) and (\ref{e6}), $\rho_{\lambda}(x)=exp_{x_{0}}(\lambda x)$,  and we consider $\beta$ a cut-off function supported in $B_{2}(x_{0})$ and equals $1$ on $B_{1}(x_{0})$. We can therefore define the functions
$$\left\{\begin{array}{ll}
u_{\lambda}=\lambda^{-\frac{1}{2}}\beta\sigma_{\lambda}^{*}(U_{0}),\\
\\
\psi_{\lambda}=\lambda^{-1}\beta\sigma_{\lambda}^{*}(\Psi_{0}),
\end{array}
\right.
$$
where $\sigma_{\lambda}=\rho_{\lambda}^{-1}$. Similar computations as in Lemma (4.5) show that
$$dE(u_{\lambda},\psi_{\lambda})\to 0, \qquad \lambda \to 0 .$$
Also we have
$$\int_{M}u_{\lambda}L_{g}u_{\lambda}dv_{g}=\int_{\R^{3}}|\nabla U_{0}|^{2}dv_{g_{\R^{3}}}+o(1),$$
$$\int_{M}\langle D_g\psi_{\lambda},\psi_{\lambda}\rangle dv_{g}=\int_{\R^{3}}\langle D_{g_{\R^3}}\Psi_{0},\Psi_{0}\rangle dv_{g_{\R^{3}}}+o(1),$$
$$\int_{M}|u_{\lambda}|^{2}|\psi_{\lambda}|^{2}dv_{g}=\int_{\R^{3}}|U_{0}|^{2}|\Psi_{0}|^{2}dv_{g_{\R^{3}}}+o(1).$$
Now, for these test functions we might have the possibility that
$$P^{-}(D_g\psi_{\lambda}-I(\psi_{\lambda})|u_{\lambda}|^{2}\psi_{\lambda})\not =0,$$
so we want to perturb $\psi_{\lambda}$ so that the previous inequality is satisfied. Therefore we have to show that there exists $h\in H^{\frac{1}{2},-}$ so that
$$P^{-}(D_g(\psi_{\lambda}+h)-I(\psi_{\lambda}+h)|u_{\lambda}|^{2}(\psi_{\lambda}+h))=0.$$
This is equivalent to solving
$$D_g h-P^{-}[I(\psi_{\lambda})|u_{\lambda}|^{2}h]+B(h)=A_{\lambda},$$
where
$$A_{\lambda}=P^{-}(D_g\psi_{\lambda}-I(\psi_{\lambda})|u_{\lambda}|^{2}\psi_{\lambda}),$$
and
$$B(h)=P^{-}\left([I(\psi_{\lambda})-I(\psi_{\lambda}+h)]|u_{\lambda}|^{2}(\psi_{\lambda}+h)\right).$$
So we consider the operator $T:H^{\frac{1}{2},-}\to H^{\frac{1}{2},-}$ defined by
$$T(h)=D_gh-I(\psi_{\lambda})P^{-}(|u_{\lambda}|^{2}h)+B(h).$$
Then
$$dT(0)\varphi=D_g\varphi-I(\psi_{\lambda})P^{-}(|u_{\lambda}|^{2}\varphi)-\langle dI(\psi_{\lambda}),\varphi\rangle P^{-}(|u_{\lambda}|^{2}\psi_{\lambda}).$$
The operator
$$\varphi\mapsto D\varphi-I(\psi_{\lambda})P^{-}(|u_{\lambda}|^{2}\varphi)$$
is negative definite on $H^{\frac{1}{2},-}$, hence it is invertible for all $\lambda>0$, and for $\lambda$ small enough we have that
$$\langle dI(\psi_{\lambda}),\varphi\rangle|u_{\lambda}|^{2}\psi_{\lambda}\to 0,$$
hence $dT(0)$ is invertible for $\lambda$ small enough. Since $A_{\lambda}\to 0$ as $\lambda \to 0$, we have by the implicit function theorem, the existence of $h_{\lambda}\in H^{\frac{1}{2},-}$ such that $T(h_{\lambda})=A_{\lambda}$, moreover $h_{\lambda}\to 0$ as $\lambda\to 0$. Now we check that
$$\int_{M}\langle D_g(\psi_{\lambda}+h_{\lambda}),(\psi_{\lambda}+h_{\lambda}\rangle dv_{g}=\int_{\R^3}\langle D_{g_{\R^3}}\Psi_{0},\Psi_{0}\rangle dv_{g_{\R^{3}}}+o(1),$$
and
$$\int_{M}|u_{\lambda}|^{2}|\psi_{\lambda}+h_{\lambda}|^{2}dv_{g}=\int_{\R^3}|U_{0}|^{2}|\Psi_{0}|^{2}dv_{g_{\R^{3}}}+o(1).$$
Hence, we have
$$\tilde{E}(u_{\lambda},\psi_{\lambda}+h_{\lambda})=Y_{g_{0}}(S^{3})\lambda_{g_{0}}^{+}(S^{3})+o(1),$$
therefore we can conclude that
$$\tilde{Y}_g(M) \leq \tilde{Y}_{g_0}(S^{3}) .$$
\end{proof}

\noindent
Now we consider the original functional $\tilde{E}(u,\psi)$ end the minimization problem giving $\tilde{Y}_g(M)$.
\noindent
We complete the proof of Theorem (\ref{second}) by proving the following
\begin{proposition}
If
$$\tilde{Y}_g(M)<\tilde{Y}_{g_{0}}(S^{3}),$$
then the problem (\ref{el}) has a non-trivial solution with $u>0$.
\end{proposition}

\begin{proof}
We introduce here a generalized Nehari manifold. We consider the functional $E$ and we notice first that there exist $t>0$ and $s>0$ small such that if $\|u\|=t$ and $\|\psi\|=s$ then $E(u,\psi)\geq c>0$. Indeed, we have that
$$2E(u,\psi)=\|u\|^{2}+\int_{M}\langle D_g\psi,\psi\rangle dv_{g}-\int_{M}|u|^{2}|\psi|^{2} dv_{g} .$$
Now using the fact that $\int_{M}\langle D\psi-I(\psi)u^{2}\psi,\varphi\rangle dv_{g}=0$, by taking $\varphi=\psi^{-}$, we have
\begin{align}
\|\psi^{-}\|^{2}&\leq I(\psi)\int_{M}|u|^{2}|\psi||\psi^{-}|dv_{g}\leq CI(\psi)\|u\|_{L^6}^{2}\|\psi\|_{L^3}\|\psi^{-}\|_{L^3},
\end{align}
hence
$$\|\psi^{-}\|\leq CI(\psi)\|u\|_{L^6}^{2}\|\psi\|_{L^3} .$$
Therefore,
$$\|\psi\|^{2}=\|\psi^{+}\|^{2}+\|\psi^{-}\|^{2}\leq \|\psi^{+}\|^{2}+CI(\psi)\|u\|_{L^6}^{4}\|\psi\|_{L^3}^{2}.$$
Now
\begin{align}
2E(u,\psi)&\geq t^{2}+s^{2}-C_{1}t^{4}s^{4}-C_{2}t^{2}s^{2}\notag\\
&\geq t^{2}+s^{2}-C(t^{8}+s^{8}+t^{4}+s^{4})\notag\\
&\geq t^{2}-C(t^{8}+t^{4})+s^{2}-C(s^{8}+s^{4}) .\notag
\end{align}
Hence for $t$ and $s$ small enough, we have that $E(u,\psi)\geq c> 0$.\\
Moreover, if we fix $u$ and $\psi$ such that $\int_{M}|u|^{2}|\psi|^{2}dv_{g}=1$, we have that
$$E(ru,r\psi)\leq r^{2}\|u\|^{2}-r^{2}\|\psi^{-}\|^{2}-r^{4}\to -\infty, \text{ when } r\to \infty.$$
This tells us that $E$ has the geometry of mountain pass, so we consider the following, min-max problem
$$m=\inf\left\{\max_{t\geq 0, s\geq 0}E(tu,s\psi);I(\psi)>0;  P^{-}\left(D_g \psi-I(\psi)u^{2}\psi\right)=0 \right\}.$$
So, if $E$ satisfies (PS) at the level $m$ and we disregard the orthogonality condition, then $m$ is a critical value. An easy computation shows that,
$$\max_{t\geq 0, s\geq 0}E(tu,s\psi)=\frac{1}{2}\tilde{E}(u,\psi)$$
Therefore $2m=\tilde{Y}_g(M)$. Now we consider the Nehari manifold
$$N=\left\{\begin{array}{ll}(u,\psi)\in H^{1}(M)\times H^{\frac{1}{2}}(\Sigma M);\int_{M}uL_{g}u dv_{g}=\int_{M}\langle D_g\psi,\psi\rangle dv_{g}=\int_{M}|u|^{2}|\psi|^{2}dv_{g} \not=0;\\
P^{-}\left(D_g \psi-I(\psi)u^{2}\psi\right)=0
\end{array}\right\} $$
We first claim that $N$ is indeed a manifold. We consider the operator
$$G:H^{1}(M)\times H^{\frac{1}{2}}(\Sigma M)\to \R\times \R \times H^{\frac{1}{2},-},$$
defined by
$$G(u,\psi)=\left[\int_{M}uL_{g}u-|u|^{2}|\psi|^{2}dv_{g},\int_{M}\langle D_g\psi-|u|^{2}\psi,\psi\rangle dv_{g},P^{-}(D_g\psi-I(\psi)|u|^{2}\psi)\right].$$
Clearly, we have that $N=G^{-1}(0)$ thus, if $DG(u,\psi)$ is onto for all $(u,\psi) \in N$ then $N$ is a manifold. So we fix $(u_{0},\psi_{0})\in N$ and we will show invertibility of $DG(u_{0},\psi_{0})$ restricted to some special subspace. Indeed, we will use the following parametrization $$(tu_{0},s\psi_{0}+h)=[t,s,h], \qquad h\in H^{\frac{1}{2},-},$$
and we will express $DG(u_{0},\psi_{0})$ in the basis of $\R\times \R\times H^{\frac{1}{2},-}$. Since $\int_{M}|u_{0}|^{2}|\psi_{0}|^{2}dv_{g}\not=0$, we will assume for the sake of simplicity that $\int_{M}|u_{0}|^{2}|\psi_{0}|^{2}dv_{g}=1$. We have then
$$DG(u_{0},\psi_{0})[1,0,0]=[0,-2,2P^{-}(D_g\psi_{0})],$$
$$DG(u_{0},\psi_{0})[0,1,0]=[-2,0,0],$$
and
$$DG(u_{0},\psi_{0})[0,0,h]=[2\langle D_g\psi_{0}, h \rangle,0,D_gh-|u_{0}|^{2}h].$$
Now we notice that the operator $K:h\to D_gh-|u_{0}|^{2}h$ is negative definite on $H^{\frac{1}{2},-}$. Indeed,
$$\langle Kh,h\rangle=-\|h\|^{2}-\int_{M}|u_{0}|^{2}h^{2}dv_{g}.$$
Therefore, $K$ is invertible. Now, given $[a,b,c]\in \R\times \R \times H^{\frac{1}{2},-}$, we want to find $[x_{1},x_{2},w]$ so that
$DG(u_{0},\psi_{0},w)=[a,b,c]$. First, we have
$$x_{1}=-\frac{b}{2},$$
and
$$2x_{1}P^{-}(D_g\psi_{0})+Kw=c.$$
Hence,
$$w=K^{-1}(c+bP^{-}(D_g\psi_{0})).$$
Finally, we have that
$$-2x_{2}+2\langle D_g\psi_{0},w\rangle =a,$$
thus
$$x_{2}=-\frac{a}{2}+\langle D_g\psi_{0},K^{-1}(c+bP^{-}(D_g\psi_{0}))\rangle.$$
This proves that $DG(u_{0},\psi_{0})$ is onto and hence $N$ is a manifold. Moreover, $DG(u_{0},\psi_{0})$ has a left inverse $A(u_{0},\psi_{0}):\R\times \R \times H^{\frac{1}{2},-}\to \R u_{0} \oplus \R \psi_{0}\oplus H^{\frac{1}{2},-}$ and $$\|A(u_{0},\psi_{0})\|_{op}\leq C(\|u_{0}\|,\|\psi_{0}\|).$$
Using Ekeland's principle now, we have the existence of a minimizing (PS) sequence for $E$ restricted to $N$: we want to show that this is indeed a (PS) sequence for $E$ also in $H^{1}(M)\times H^{\frac{1}{2}}(\Sigma M)$. Let us call such a sequence $(u_{n},\psi_{n})\in N$. Then similarly as in the proof of Proposition (5.1), we set $DE(u_{n},\psi_{n})=\varepsilon_{n}$. We clearly have that $\varepsilon_{n}^{T}\to 0$, since it is the tangential part of the (PS) sequence and it is a (PS) sequence in $N$. Notice now that
$$P_{n}=A(u_{n},\psi_{n})\circ DG(u_{n},\psi_{n})$$
is a projector on $\R u_{0} \oplus \R \psi_{0}\oplus H^{\frac{1}{2},-}$ parallel to $T_{(u_{0},\psi_{0})}N$. Now since
$$E(u_{n}\psi_{n})=\frac{1}{2}\int_{M}|u_{n}|^{2}|\psi_{n}|^{2}dv_{g}\to m,$$
we have that $\|u_{n}\|^{2}\leq C$. Also, we have
$$-\|\psi^{-}_{n}\|^{2}=\int_{M}|u_{n}|^{2}\langle \psi_{n},\psi_{n}^{-}\rangle dv_{g}.$$
Thus
\begin{align}
\|\psi^{-}_{n}\|^{2}&\leq \int_{M}|u_{n}|^{2}|\psi_{n}||\psi_{n}^{-}| dv_{g}\notag\\
&\leq \left(\int_{M}|u_{n}|^{2}|\psi_{n}|^{2}dv_{g}\right)^{\frac{1}{2}}\left(\int_{M}|u_{n}|^{2}|\psi_{n}^{-}|^{2}dv_{g}\right)^{\frac{1}{2}}\notag\\
&\leq C_{1}\|u_{n}\|_{L^6}\|\psi_{n}^{-}\|_{L^3}.\notag
\end{align}
Therefore
$$\|\psi^{-}_{n}\|\leq C,$$
but we have that
$$\|\psi_{n}^{+}\|^{2}-\|\psi_{n}^{-}\|^{2}=\int_{M}|u_{n}|^{2}|\psi_{n}|^{2}dv_{g},$$
hence
$$\|\psi_{n}^{+}\|^{2}\leq C.$$
Therefore, we have that $A(u_{n},\psi_{n})$ is uniformly bounded and so is $P_{n}$. We consider now the operator $P_{n}^{*}$, the adjoint of $P_{n}$. Then $P_{n}^{*}$ is also a projector on $(\R u_{0} \oplus \R \psi_{0}\oplus H^{\frac{1}{2},-})^{\perp}$ parallel to $\mathcal{N}_{(u_{n},\psi_{n})}N$, the normal space of $N$ at the point $(u_{n},\psi_{n})$. We also notice that
$$\varepsilon_{n}\in (\R u_{0} \oplus \R \psi_{0}\oplus H^{\frac{1}{2},-})^{\perp},$$
hence, $\varepsilon_{n}=P_{n}^{*}\varepsilon_{n}^{T}$ and so $(u_{n},\psi_{n})$ is indeed a (PS) sequence for $E$. Therefore, this (PS) sequence is at the energy level $\frac{1}{2}\tilde{Y}_g(M)$: from Theorem (\ref{first}), if $\tilde{Y}_g(M)<\tilde{Y}_{g_{0}}(S^{3})$, then this (PS) sequence converges and thus we have a solution to our problem.
\end{proof}


\section{Existence of infinitely many solutions in symmetric manifolds}

\noindent
Here we will consider a three-dimensional closed manifold $(M,g)$ with an isometric group action $G$ acting on $M$, such that the orbits of $G$ have infinite cardinality. As an example, we can consider the standard sphere $S^{3}\subseteq\R^4$ with the action introduced by Ding \cite{Ding}, that is $G=O(2)\times O(2)$. Such symmetries where exploited an improved in other settings such as in \cite{M2,MV2,MVG}.  We will show the following
\begin{theorem}\label{third}
Given a manifold $M$ as described above, then $(\ref{el})$ has two infinite families of solutions.
\end{theorem}
\begin{proof}
First of all we notice that the functional $E$ satisfies the (PS) condition on the space $H_G:=H_{G}^{1}(M)\times H_{G}^{\frac{1}{2}}(\Sigma M)$,
where $H_{G}^{1}(M)$ and $H_{G}^{\frac{1}{2}}(\Sigma M)$ are respectively the subspaces of  $H^{1}(M)$ and $H^{\frac{1}{2}}(\Sigma M)$ which are invariant under the action of $G$. In order to prove this claim, let us consider $z_{n}\in H_{G}$ a (PS) sequence for $E$, then according to the characterization in Theorem (\ref{first}) above we have that
$$E(z_{n})=E(z_{\infty})+\sum_{k=1}^{m}c_{k}+o(1) ,$$
where $c_{k}=E_{\R^3}(Z_\infty^{k})\geq \tilde{Y}(S^{3})$, and $Z_\infty^{k}$ are solutions of equation (\ref{eqR3}) in $\R^3$. The main point is that the number of these solutions is finite and that the energy is finite. In particular if $\{z_{n}\}$ is a (PS) sequence that concentrates on $x^{1},\cdots, x^{m}$ then $z_{n}(h\cdot)$ concentrates at $h\cdot x^{1}, \cdots, h\cdot x^{m}$ for every $h\in G$. Now, since $z_{n}\in H_{G}$, then $z_{n}(h\cdot)=z_{n}$ hence $z_{n}$ concentrates at all the orbits of $x^{1},\cdots, x^{m}$ under the action of $G$; but the orbits are infinite: therefore the set of concentration needs to be empty and hence the (PS) condition holds.\\
Now we consider the functional $E:H_{G}\to \mathbb{R}$ defined by
$$E(u,\psi)=\frac{1}{2}\left(\int_{M}uL_{g}udv_g+\int_{M}\langle D_g\psi,\psi\rangle dv_g -\int_{M}|u|^{2}|\psi|^{2}dv_g\right).$$
We will study the restriction of this functional to the Nehari manifold $N_G$ defined by
$$N_G=\left\{\begin{array}{ll}(u,\psi)\in H_{G};\int_{M}uL_{g}udv=\int_{M}|u|^{2}|\psi|^{2}dv_g=\int_{M}\langle D_g\psi,\psi\rangle dv_g\neq0;\\
P^{-}(D_g\psi-I(\psi)|u|^{2}\psi)=0\end{array}\right\}.$$
As in the previous section, $N_G$ is a manifold, moreover critical points of $E_{|N_G}$ are critical points of $E$, as we saw above, and moreover any (PS) sequence of $E_{|N_G}$ is a (PS) sequence of $E$. Therefore, $E_{|N_G}$ satisfies the (PS) condition. So now we want to use the classical min-max theorem on the manifold $N_g$, so we define a collection $\mathcal{A}$ of sets $A\subset N_G$ such that $-A=A$ and
$$c_{k}=\inf_{A\in \mathcal{A};\gamma(A)\geq k}\max_{(u,\psi)\in A}E(u,\psi),$$
where $\gamma(A)$ denotes the genus of $A$. Now, if we can show that $N_G$ contains sets of arbitrarily high genus, we can show that we have infinitely many solutions. To this aim, we will prove that there exists a continuous $\mathbb{Z}_{2}$-equivariant map
$$\mathcal{T}:\{-1,1\}\times S^{+} \longrightarrow N_G,$$
where $S^{+}$ is the unit sphere of $H^{\frac{1}{2},+}$. First, we recall that the generalized Nehari manifold originates from considering the functional $R:\mathbb{R}^{+}\times \mathbb{R}^{+}\times H^{\frac{1}{2},-}\to \mathbb{R}$, defined by
$$R(t,s,\varphi)=E(tu,s(\psi+\varphi)).$$
Therefore, the nonzero critical points of $R$ are in $N_G$. Indeed, if such a critical point exists, then it satisfies
$$\left\{\begin{array}{lll}
s^{2}=\frac{\displaystyle\int_{M}uL_{g}udv_g}{\displaystyle\int_{M}|u|^{2}|\psi+\varphi|^{2}dv_g},\\
\\
t^{2}=\frac{\displaystyle\int_{M}\langle D_g(\psi+\varphi),\psi+\varphi)dv_g}{\displaystyle\int_{M}|u|^{2}|\psi+\varphi|^{2}dv_g},\\
\\
P^{-}(D_g(\psi+\varphi)-I(\psi+\varphi)|u|^{2}(\psi+\varphi))=0.\\
\end{array}
\right.
$$
Now, the main issue in solving this system resides in the last equation, which is equivalent to solving
$$T(\varphi)+B(\varphi)=A(\varphi),$$
where $$T(\varphi)=P^{-}(D_g\varphi-I(\psi)|u|^{2}\varphi),$$
$$B(\varphi)=P^{-}([I(\psi)-I(\psi+\varphi)]|u|^{2}(\psi+\varphi)),$$
and
$$A(\varphi)=P^{-}(D_g\psi-I(\psi)|u|^{2}\varphi).$$
Again, as in the previous section, the operator $T$ is invertible, so the term that we need to consider here is $B(\varphi)$. Now, we notice that for some particular choice of $u$ and $\psi$ we can always find a solution to this system. Indeed, if $u$ is constant and $\psi\in H^{\frac{1}{2},+}$, then we can take $\varphi=0$, so that we have a unique critical point of $R$ denoted by $(t_{(u,\psi)},s_{(u,\psi)},0)$ such that $(t_{(u,\psi)}u,s_{(u,\psi)}(\psi))\in N_G$. Therefore, we will consider the map $\mathcal{T}$ defined by $$\mathcal{T}(1,\psi)=(t_{(1,\psi)},s_{(1,\psi)}\psi),$$
where
$$t_{(1,\psi)}=\frac{1}{\|\psi\|_{L^{2}}},$$
and
$$s_{(1,\psi)}=\frac{\displaystyle\frac{1}{8}\int_{M}R_g dv_g}{\|\psi\|_{L^{2}}}.$$
Clearly $\mathcal{T}(-(\delta,\psi))=-\mathcal{T}(\delta,\psi)$ where $\delta=\pm 1$ and the map $\mathcal{T}$ is continuous. Now, since $\mathcal{T}$ is an equivariant map, and since $S^{+}$ has infinite genus, that is $\gamma(S^{+})=+\infty$, we have also that $N_G$ has infinite genus; moreover if $A\subset S^{+}$ is symmetric and such that $\gamma(A)=k$, then $\mathcal{T}(A)\subset N_G$ satisfies $\gamma(\mathcal{T}(A))\geq k$. Also since $E$ is bounded from below on $N_G$, we have by classical min-max argument, see \cite{Rab}, that $E_{|N_G}$ has infinitely many critical points, hence $E$ has infinitely many critical points.\\
Finally, in order to find another infinite family of solutions, we argue in a similar way, by noticing that the set $N_{G}$ is invariant under the action of $S^{1}$ on the spinorial part, defined by
$$\theta \cdot (u,\psi)=(u,e^{i2\pi \theta} \psi).$$
Clearly, $E_{|N_G}$ is also invariant under the this action of $S^{1}$. Therefore we can define the family of sets $\mathcal{K}$ by saying that a set $A$ belongs to $\mathcal{K}$ if and only if $ e^{i2\pi \theta}A=A$. We define also the min-max levels
$$\tilde{c}_{k}=\inf_{A\in \mathcal{K};i_{S^{1}}(A)\geq k}\max_{(u,\psi)\in A}E(u,\psi),$$
where $i_{S^{1}}$ is the Faddell-Rabinowitz cohomological index \cite{Fa}. Then, we use a restriction of the previous map $\mathcal{T}$, that we denote here by $\mathcal{G}:S^{+}\to N_G$, defined by
$$\mathcal{G}(\psi)=\mathcal{T}(1,\psi).$$
We see that $\mathcal{G}$ is $S^{1}$-equivariant, hence $i_{S^{1}}(N_G)=+\infty$ and hence, $E_{|N_G}$ has infinitely many critical points.
\end{proof}

\end{document}